\documentclass{amsart}

\usepackage[colorlinks=true, urlcolor=blue, citecolor=blue, linkcolor=blue, hyperfootnotes=true]{hyperref}
\usepackage{amssymb}
\usepackage{mathtools}
\usepackage{tikz-cd}
\usepackage{pifont}
\newcommand{\cmark}{\ding{51}}%
\newcommand{\xmark}{\ding{55}}%
\usepackage{aliascnt}

\numberwithin{equation}{section}

\newtheorem{thm}{Theorem}[section]

\newaliascnt{prp}{thm}
\newtheorem{prp}[prp]{Proposition}
\aliascntresetthe{prp}

\newaliascnt{cor}{thm}
\newtheorem{cor}[cor]{Corollary}
\aliascntresetthe{cor}

\theoremstyle{definition}

\newaliascnt{dfn}{thm}
\newtheorem{dfn}[dfn]{Definition}
\aliascntresetthe{dfn}

\newaliascnt{qst}{thm}
\newtheorem{qst}[qst]{Question}
\aliascntresetthe{qst}

\newaliascnt{ex}{thm}
\newtheorem{ex}[ex]{Example}
\aliascntresetthe{ex}

\newcommand{\C}{\mathrel{\Cap}}
\newcommand{\pp}{\mathrel{\prec\hspace{-4pt}\prec}}
\newcommand\subsetsim{\mathrel{\substack{\textstyle\subset\\[-0.2ex]\textstyle\sim}}}
\newcommand\ssubset{\mathrel{\mathrlap{\subset}\hphantom{\ll}\mathllap{\subset}}}

\author{Tristan Bice}
\email{Tristan.Bice@gmail.com}
\thanks{The first author is supported by IMPAN through a WCMCS postdoctoral grant}
\author{Charles Starling}
\thanks{The second author is supported by the NSERC grants of Beno\^it Collins, Thierry Giordano, and Vladimir Pestov.}
\thanks{This collaboration began at the Fields Institute ``Workshop on Dynamical Systems and Operator Algebras'' and ``Workshop on New Directions in Inverse Semigroups'' held at the University of Ottawa in May-June 2016.}

\title{Locally Compact Stone Duality}

\begin{document}

\begin{abstract}
We prove a number of dualities between posets and (pseudo)bases of open sets in locally compact Hausdorff spaces.  In particular, we show that
\begin{enumerate}
\item Relatively compact basic sublattices are finitely axiomatizable.
\item Relatively compact basic subsemilattices are those omitting certain types.
\item Compact clopen pseudobasic posets are characterized by separativity.
\end{enumerate}
We also show how to obtain the tight spectrum of a poset as the Stone space of a generalized Boolean algebra that is universal for tight representations.
\end{abstract}

\maketitle

\section*{Introduction}

\subsection*{Background}

A number of dualities exist between classes of lattices and topological spaces.  Those most relevant to the present paper are summarized below.

\vspace{-10pt}
\begin{figure}[h]
\caption{Dualities}\label{Dualities}
\vspace{5pt}\begin{tabular}{|r|c|c|c|c|r|}
\hline Topology\hspace{50pt} & Lattice & $1^\mathrm{st}$-Ord & $\vee=\cup$ & Reference\\ \hline
$0$-Dimensional Compact Hausdorff & Boolean & \cmark & \cmark & \cite{Stone1936}\\
Compact Hausdorff & Normal Disjunctive & \cmark & \cmark & \cite{Wallman1938}\\
Locally Compact Hausdorff & $R$-Lattice & \cmark & \xmark & \cite{Shirota1952}\\
Locally Compact Sober\hspace{18pt} & Continuous Frame & \xmark & \cmark & \cite{HofmannLawson1978}\\ \hline
\end{tabular}
\end{figure}
The most well studied dualities are the first and the last.  Indeed, Boolean algebras have a very long history and there is also a considerable amount of literature on both domains/continuous lattices (see \cite{GierzHofmannKeimelLawsonMisloveScott2003}) and locales/frames (see \cite{PicadoPultr2012}).  However, one key feature of the former which is not shared by the latter is that Boolean algebras are first order structures.  More precisely, Boolean algebras are defined by a finite list of first order sentences in a language with a single binary relation $\leq$.  On the other hand, both domains and frames require some degree of completeness, which requires quantification over subsets, making these second order rather than first order structures.  Moreover, domains are defined by the way-below relation, while frames require infinite distributivity, both of which are also undeniably second order.  The unfortunate consequence of this is that classical first order model theory can not be applied to domains or frames as it is to Boolean algebras.

Furthermore, while the duality between sober spaces and spatial frames has its origins in Stone duality (see \cite{Johnstone1986}), point-free topology is more accurately described as a close analogy rather than a direct generalization of Stone duality.  Indeed, even for zero-dimensional $X$, the entire open set lattice $\mathcal{O}(X)$ is much larger than its clopen sublattice.  In fact, $\mathcal{O}(X)$ is uncountable for any infinite Hausdorff $X$, which explains why frames have no first order description, as this would contradict the downward L\"{o}wenheim-Skolem theorem.

While less well known, \cite{Shirota1952}\footnote{We would like to thank Tomasz Kania for directing our attention to this article.} deals with both of these issues, at least for locally compact Hausdorff spaces.  Indeed, \cite[Definition 2]{Shirota1952} describes $R$-lattices as those satisfying a finite list of first order sentences, albeit in a language with two relations $\leq$ and $\ll$ and ternary function implicit in part v) (although we believe a more careful axiomatization could be given just in terms of $\ll$).  Moreover, the $R$-lattices where $\ll$ is reflexive and which also have a maximum are precisely the Boolean algebras, so \cite[Theorem 1]{Shirota1952} is a direct generalization of Stone's original duality.  The only issue here is that $R$-lattices represent relatively compact basic (i.e. forming a basis in the usual topological sense) sublattices of $\mathcal{RO}(X)(=$ \emph{regular} open subsets of $X$) rather than $\mathcal{O}(X)$.  This means joins are not unions, specifically $O\vee N=\overline{O\cup N}^\circ$ rather than $O\cup N$.  Alternatively, we could consider the lattice elements as representing regular closed sets instead, but then $O\wedge N=\overline{(O\cap N)^\circ}$ rather than $O\cap N$.  As topological properties are usually expressed in terms of $\cup$ and $\cap$, this makes $R$-lattices somewhat less appealing for doing first order topology.

\subsection*{Outline}

Our first goal is thus to modify the axioms of $R$-lattices (see \autoref{BasicLatticeDfn}) so as to axiomatize relatively compact basic sublattices of $\mathcal{O}(X)$ instead.  This is the content of \autoref{BL}-\autoref{SS}, as summarized in \autoref{basiclat}.  We then extend this to an equivalence of categories in \autoref{I}, taking appropriate relations as our basic lattice morphisms (see \autoref{categoryequiv}).  Next we consider relatively compact basic meet subsemilattices of $\mathcal{O}(X)$ in \autoref{BS}.  Here a finite axiomatization is not possible, as explained at the end of \autoref{SS}, however we show that they can still be characterized by omitting types, as summarized in \autoref{semichar}.

In \autoref{T}-\autoref{TTS}, we consider a different generalization of classic Stone duality where we extend the bases rather than the spaces under consideration.  Specifically, we consider `pseudobases' (see \autoref{pseudobasisdef}) of compact clopen sets in (necessarily) zero-dimensional locally compact Hausdorff $X$.  Here we show that p0sets(=posets with minimum $0$) which arise from such pseudobases can be finitely axiomatized just by separativity (which goes under various names \textendash\, see the note at the end of \autoref{TTS}), and that $X$ can still be recovered from any compact clopen pseudobasis as its tight spectrum, as summarized in \autoref{pseudochar}.  Also, we use a well known set theoretic construction to define a generalized Boolean algebra from any p0set that is universal for tight representations (see \autoref{Universality}).  This allows us to identify the tight spectrum of the p0set with the Stone space of the algebra, providing a different take on some of the theory from \cite{Exel2008}.

\subsection*{Related Work}

Given the classical nature of the results in \autoref{Dualities}, it seems that a description of these dualities is well overdue.  Previous literature certainly gets close describing basic lattices, as most of the axioms already appear in \cite{Shirota1952} and a first order analog of the way-below relation was already considered in \cite{Johnstone1986} \textendash\, the only extra step really needed was to consider it as a fully-fledged replacement (see the comments after \autoref{basiclatticeproperties}).  Also, the idea of representing continuous functions by certain relations, as in \autoref{I}, already appears in formal topology (see \cite{CirauloMaiettiSambin2013}), where constructions similar to those in \autoref{BS} also appear.  There have also been various other extensions of Stone duality, often in the context of categorical or continuous rather than classical logic and/or based on the ring structure of $C(X,\mathbb{C})$, the lattice structure of $C(X,\mathbb{R})$ or the MV-algebra structure of $C(X,[0,1])$ (see \cite{KaniaRmoutil2016}, \cite{MarraReggio2015}, \cite{Russo2016} and the references therein).

There are also non-commutative extensions, e.g. classic Gelfand duality allows us to see C*-algebras as non-commutative locally compact Hausdorff spaces, which have recently been investigated from a continuous model theoretic point of view (see \cite{FarahHartSherman2014} and \cite{FHLRTVW2016}).  Inverse semigroups provide a different non-commutative generalization (see \cite{KudryavtsevaLawson2016}), although they are closely related (see \cite{Exel2008}).  In fact, our original motivation was to define combinatorial C*-algebras from inverse semigroups in a more general way to include C*-algebras with few projections.  The present paper can be viewed as the commutative case \textendash\, we hope to elaborate on the non-commutative case in forthcoming papers.

\section{Basic Lattices}\label{BL}

Assume $\prec$ is a transitive relation on set $B$ with minimum $0$, i.e.
\begin{align}
\label{Minimum}\tag{Minimum}\forall x\ (&0\prec x).\hspace{70pt}\\
\label{Transitivity}\tag{Transitivity}x\prec y\prec z\quad\Rightarrow\quad& x\prec z.\\
\intertext{Define a preorder $\preceq$ and symmetric relations $\perp$ and $\Cap$ by}
\label{Reflexivization}\tag{Reflexivization}x\preceq y\quad\Leftrightarrow\quad&\forall z\prec x\ (z\prec y).\\
\label{Intersects}\tag{Intersects}x\C y\quad\Leftrightarrow\quad&\exists z\neq 0\ (z\prec x,y).\\
\label{Disjoint}\tag{Disjoint}x\perp y\quad\Leftrightarrow\quad&\nexists z\neq 0\ (z\prec x,y).\\
\intertext{Note that the definition of $\preceq$ and the transitivity of $\prec$ immediately yields}
\label{LeftAuxiliarity}\tag{Left Auxiliarity}x\prec z\preceq y\quad\Rightarrow\quad& x\prec y.\\
\label{Domination}\tag{Domination}x\prec y\quad\Rightarrow\quad& x\preceq y.
\end{align}

\begin{dfn}\label{BasicLatticeDfn}
We call $(B,\prec)$ a \emph{basic lattice} if $(B,\preceq)$ is also a lattice and
\begin{align}
\label{Coinitiality}\tag{Coinitiality}\forall x\neq0\ &(x\mathrel{\Cap}x).\\
\label{Cofinality}\tag{Cofinality}\forall x\ \exists y\ &(x\prec y).\\
\label{Interpolation}\tag{Interpolation}x\prec y\quad\Rightarrow\quad&\exists z\ (x\prec z\prec y).\\
\label{Multiplicativity}\tag{Multiplicativity}x\prec x'\ \&\ y\prec y'\quad\Rightarrow\quad&x\wedge y\,\prec\,x'\wedge y'.\\
\label{Additivity}\tag{Additivity}x\prec x'\ \&\ y\prec y'\quad\Rightarrow\quad&x\vee y\,\prec\,x'\vee y'.\\
\label{Decomposition}\tag{Decomposition}z\prec x\vee y\quad\Rightarrow\quad&\exists x'\prec x\ \exists y'\prec y\ (z=x'\vee y').\\
\label{Complementation}\tag{Complementation}x\prec y\prec z\quad\Rightarrow\quad&\exists w\perp x\ (w\vee y=z).
\end{align}
\end{dfn}

As part of the definition of a lattice we take it that $\preceq$ is antisymmetric and hence a partial order.  Also, by \eqref{Coinitiality}, \eqref{LeftAuxiliarity} and \eqref{Domination}, we could replace $\prec$ with $\preceq$ in the definitions of $\C$ and $\perp$, i.e. for basic lattice $B$
\begin{align*}
x\C y\qquad&\Leftrightarrow\qquad x\wedge y\neq 0.\\
x\perp y\qquad&\Leftrightarrow\qquad x\wedge y=0.
\end{align*}

Most of the axioms above already appear in some form in \cite[Definition 2]{Shirota1952}.  The key extra axiom is \eqref{Decomposition}, which only applies to open set lattices, not regular open set lattices (the key axiom omitted is \eqref{Separativity} mentioned below, which only applies to regular open set lattices, not open set lattices).  Indeed, let $X$ be a locally compact Hausdorff space, so $X$ has a basis $B\subseteq\mathcal{O}(X)(=$ open subsets of $X)$ of relatively compact sets, i.e.
\begin{enumerate}
\item $\emptyset\in B\subseteq\mathcal{O}(X)$.
\item $\forall O\in\mathcal{O}(X)\ \forall x\in O\ \exists N\in B\ (x\in N\subseteq O)$.
\item $\overline{O}$ is compact, for all $O\in B$.
\end{enumerate}
Define $\subset$ on $B$ by
\[\label{waybelowdef}\tag{Compact Containment}O\subset N\qquad\Leftrightarrow\qquad\overline{O}\subseteq N.\]
Note then inclusion $\subseteq$ is indeed the relation defined from $\subset$ by \eqref{Reflexivization}.

\begin{prp}\label{BasicBasis}
If $B$ is a basis of relatively compact open sets that is closed under $\cup$ and $\cap$ then $(B,\subset)$ is a basic lattice.
\end{prp}

\begin{proof} We prove the last two properties and leave the rest as an exercise.
\begin{itemize}
\item[\eqref{Decomposition}]  If $O\subset M\cup N$ then, for each $x\in\overline{O}$, we have $O_x\in B$ with $O_x\subset M$ or $O_x\subset N$.  As $\overline{O}$ is compact, $(O_x)$ has a finite subcover $F$.  Let
\[M'=O\cap\bigcup_{\substack{O'\in F\\ O'\subset M}}O'\qquad\text{and}\qquad N'=O\cap\bigcup_{\substack{O'\in F\\ O'\subset N}}O'.\]
Then $M'\subset M$, $N'\subset N$ and $O=M'\cup N'$, as required.

\item[\eqref{Complementation}]  If $M\subset N\subset O$ then, for each $x\in\overline{O}\backslash N$, we have $O_x\in B$ with $O_x\cap\overline{M}=\emptyset$.  As $\overline{O}\backslash N$ is compact, $(O_x)$ has a finite subcover $F$.  Letting $L=O\cap\bigcup F$, we have $L\cap M=\emptyset$ and $L\cup N=O$.  \qedhere
\end{itemize}
\end{proof}

In the next section we show that all basic lattices arise in this way from locally compact Hausdorff spaces.  Here we just note some more properties of basic lattices.

\begin{prp}\label{basiclatticeproperties}
Any basic lattice $B$ satisfies the following.
\begin{align}
\label{Distributivity}\tag{Distributivity}z\preceq x\vee y\quad\Leftrightarrow\quad& z\preceq(x\wedge z)\vee(y\wedge z).\\
\label{RatherBelow}\tag{Rather Below}x\prec y\quad\Leftrightarrow\quad&\forall z\ \exists w\perp x\ (z\preceq w\vee y).
\end{align}
\end{prp}

\begin{proof}\
\begin{itemize}
\item[\eqref{Distributivity}]  $\Leftarrow$ is immediate.  Conversely, by \eqref{LeftAuxiliarity} and \eqref{Decomposition}, for any $w\prec z$, we have $x'\prec x$ and $y'\prec y$ with $w=x'\vee y'$.  By \eqref{Multiplicativity}, $x'=x'\wedge w\prec x\wedge z$ and $y'=y'\wedge w\prec y\wedge z$.  By \eqref{Additivity}, $w=x'\vee y'\prec(x\wedge z)\vee(y\wedge z)$.  As $w\prec z$ was arbitrary, $z\preceq(x\wedge z)\vee(y\wedge z)$.
\item[\eqref{RatherBelow}]  For the $\Rightarrow$ part, assume $x\prec y$.  By \eqref{Cofinality}, we have $y'\succ y$.  By \eqref{LeftAuxiliarity}, $y\prec y'\vee z$, for any $z\in B$.  By \eqref{Complementation}, we have $w\perp x$ with $w\vee y=y'\vee z\succeq z$, as required.

Conversely, take $x,y\in B$ satisfying the right hand side.  By \eqref{Cofinality}, we have $z\succ x$ and then we can take $w\perp x$ with $z\preceq w\vee y$.  By \eqref{LeftAuxiliarity}, $x\prec w\vee y$.  By \eqref{Decomposition}, we have $y'\prec y$ and $w'\prec w$ with $x=w'\vee y'$.  By \eqref{Domination}, $w'\preceq w\perp x$ so $w'=w'\wedge x=0$ and hence $x=y'\prec y$. \qedhere
\end{itemize}
\end{proof}

When $B$ has a maximum $1$, it suffices to take $z=1$ in \eqref{RatherBelow} which is the definition of the \emph{rather below} relation in \cite[Ch 5 \S5.2]{PicadoPultr2012} and the \emph{well inside} relation in \cite[III.1.1]{Johnstone1986}.  In any case, \autoref{basiclatticeproperties} shows that we could equivalently take $\preceq$ as the primitive relation in the definition of a basic lattice and define $\prec$ from $\preceq$ as in \eqref{RatherBelow}.\footnote{This contrasts with $R$-lattices in \cite{Shirota1952}, where $\leq$ can be defined from $\ll$ but not vice versa.}  Then the definition of $\preceq$ from $\prec$ at the start would become a defining property of a basic lattice instead.  Indeed, most treatments of continuous lattices take $\preceq$ as the primary notion and define $\prec$ as the way-below relation from $\preceq$, but we will soon see that there are good reasons to focus more on $\prec$.

The basic lattice axioms could also be reformulated in several ways.  For example, we could combine \eqref{Cofinality} and \eqref{Complementation} into
\[\label{<Below}\tag{$\prec$-Below}x\prec y\quad\Rightarrow\quad\forall z\ \exists w\perp x\ (z\prec w\vee y).\]
We could also replace $\Rightarrow$ with $\Leftrightarrow$ in \eqref{Decomposition} to combine it with \eqref{Additivity}.  Or we could replace \eqref{Decomposition} and \eqref{Interpolation} with \eqref{Distributivity} and
\[\label{vInterpolation}\tag{$\vee$-Interpolation}z\prec x\vee y\quad\Rightarrow\quad\exists x'\prec x\ \exists y'\prec y\ (z\prec x'\vee y').\]
Or we could even avoid the use of meets in \eqref{Multiplicativity}, as shown below.

\begin{prp}\label{nomeets}
If $B$ is a lattice satisfying \eqref{Cofinality} then \eqref{Interpolation},\\ \eqref{Multiplicativity} and \eqref{Additivity} are equivalent to the following.
\begin{align}
\label{RightAuxiliarity}\tag{Right Auxiliarity}x\preceq z\prec y\quad&\Rightarrow\quad x\prec y.\\
\label{RieszInterpolation}\tag{Riesz Interpolation}x,x'\prec y,y'\quad&\Rightarrow\quad\exists z\ (x,x'\prec z\prec y,y').
\end{align}
\end{prp}

\begin{proof}
If $B$ is lattice satisfying \eqref{Cofinality} and $x\preceq y\prec z$ then we have some $w\succ x$.  If $B$ satisfies \eqref{Multiplicativity} then $x=x\wedge y\prec w\wedge z\preceq z$ so $x\prec z$, by \eqref{LeftAuxiliarity}, i.e. \eqref{RightAuxiliarity} holds.  If $B$ also satisfies \eqref{Additivity} then $x,x'\prec y,y'$ implies $x\vee x'\prec y\wedge y'$.  Further assuming $B$ satisfies \eqref{Interpolation}, we have $z\in B$ with $x\vee x'\prec z\prec y\wedge y'$ and hence $x,x'\prec z\prec y,y'$, by \eqref{LeftAuxiliarity} and \eqref{RightAuxiliarity}, i.e. \eqref{RieszInterpolation} holds.

Conversely, assuming \eqref{RieszInterpolation} and \eqref{RightAuxiliarity}, if $x\prec x'$ and $y\prec y'$ then $x\wedge y\prec x',y'$ and hence $x\wedge y\prec z\prec x',y'$, for some $z\in B$.  By \eqref{Domination}, $z\preceq x',y'$ and hence $z\preceq x'\wedge y'$.  By \eqref{LeftAuxiliarity}, $x\wedge y\prec x'\wedge y'$ so we have \eqref{Multiplicativity}.\footnote{See \cite[Lemma I-3.26]{GierzHofmannKeimelLawsonMisloveScott2003} for this and other characterizations of \eqref{Multiplicativity}.}  In the same way we get \eqref{Additivity}.
\end{proof}

Lastly, let us note that when $\prec$ is reflexive, i.e. when $\prec$ coincides with $\preceq$, most of the basic lattice axioms are automatically satisfied.  Indeed, for a lattice $(B,\preceq)$ to be a basic lattice it need only satisfy \eqref{Decomposition}, which is then the same as \eqref{Distributivity}, and \eqref{Complementation} which, as we can take $x=y$, is saying that $B$ is \emph{section complemented} in the terminology of \cite{Stern1999}.  In other words,
\[(B,\preceq)\text{ is a basic lattice}\qquad\Leftrightarrow\qquad(B,\preceq)\text{ is a generalized Boolean algebra}.\]
So while we have no unary complement operation like in a true Boolean algebra, we do have a binary relative complement operation $x\backslash y$, i.e. satisfying
\[(x\backslash y)\wedge(x\wedge y)=0\qquad\text{and}\qquad(x\backslash y)\vee(x\wedge y)=x.\]

\section{Filters}

\begin{dfn}
For any transitive relation $\prec$ on $B$, we call $U\subseteq B$ a \emph{$\prec$-filter} if
\begin{align}
\label{UpwardsClosed}\tag{$\succ$-Closed}x\succ y\in U\qquad&\Rightarrow\qquad x\in U.\\
\label{DownwardsDirected}\tag{$\prec$-Directed}x,y\in U\qquad&\Rightarrow\qquad\exists z\in U\ (z\prec x,y).
\end{align}
We call $U\subsetneqq B$ a \emph{$\prec$-ultrafilter} if $U$ is maximal among proper $\prec$-filters.
\end{dfn}

Throughout the rest of this section,
\[\textbf{$B$ is an arbitrary but fixed basic lattice.}\]
We call $U\subseteq B$ \emph{$\prec$-coinitial}\footnote{In \cite[VII.4.2]{PicadoPultr2012}, ideals satisfying the dual notion are called \emph{regular} ideals, but we use regular later in a different sense closer to the usual notion of a regular open set.} if $U\subseteq U^\prec$ where
\[U^\prec=\{y\in B:\exists x\in U\ (x\prec y)\}.\]
So \eqref{Coinitiality} is just saying that $B\backslash\{0\}$ is $\prec$-coinitial.

\begin{prp}\label{preceqfilters}
If $U\subseteq B$ is $\preceq$-directed then $U^\prec$ is a $\prec$-filter.  Moreover
\begin{equation}\label{preccoinit}
U\text{ is a $\prec$-filter}\qquad\Leftrightarrow\qquad U\text{ is a $\prec$-coinitial $\preceq$-filter}.
\end{equation}
\end{prp}

\begin{proof}
As $\prec$ is transitive, $U^\prec$ is \eqref{UpwardsClosed}.  If $U\ni x\prec x'$ and $U\ni y\prec y'$ then we have some $z\in U$ with $z\preceq x,y$, as $U$ is $\preceq$-directed.  By \eqref{LeftAuxiliarity}, $z\prec x',y'$ so, by \eqref{RieszInterpolation}, we have $z'\in B$ with $U\ni z\prec z'\prec x',y'$, i.e. $U^\prec$ is $\prec$-directed and hence a $\prec$-filter.

For the $\Rightarrow$ part of \eqref{preccoinit}, note any $\prec$-directed $U\subseteq B$ is $\prec$-coinitial and, by \eqref{Domination}, $\preceq$-directed.  And any $\prec$-coinitial $\succ$-closed $U\subseteq B$ is $\succeq$-closed, as $x\succeq y\in U$ implies $y\succ z\in U$, for some $z$, so $x\succ z\in U$, by \eqref{LeftAuxiliarity}, and hence $x\in U$.

Conversely, for the $\Leftarrow$ part of \eqref{preccoinit}, note any $\succeq$-closed $U\subseteq B$ is $\succ$-closed, by \eqref{Domination}.  And any $\prec$-coinitial $\preceq$-directed $U\subseteq B$ is $\prec$-directed, as any $x,y\in U$ satisfies $z\preceq x,y$, for some $z\in U$, and then $w\prec z$, for some $w\in U$, and hence $w\prec x,y$, again by \eqref{LeftAuxiliarity}.
\end{proof}

Ultrafilters in Boolean algebras can be characterized in a couple of first order ways as the proper prime filters or as the proper filters that intersect every complementary pair.  These characterizations generalize to basic lattices as follows.

We call a $\succ$-filter a \emph{$\prec$-ideal}.

\begin{prp}\label{ultrachars}
For non-empty $\prec$-filter $U\subsetneqq B$, the following are equivalent.
\begin{enumerate}
\item\label{ultra} $U$ is a $\prec$-ultrafilter.
\item\label{prime} $B\setminus U$ is a $\preceq$-ideal.
\item\label{compy} $B\setminus U=\{y\in B:\forall x\prec y\ \exists w\in U(w\perp x)\}$.
\end{enumerate}
\end{prp}

\begin{proof}\
\begin{itemize}
\item[\eqref{ultra}$\Rightarrow$\eqref{compy}]  For any proper $\prec$-filter $U\subseteq B$,
\[\{y\in B:\forall x\prec y\ \exists w\in U(w\perp x)\}\subseteq B\setminus U.\]
For the reverse inclusion, assume $y$ is not in the set on the left, so we have $x\prec y$ such that $w\Cap x$, for all $w\in U$.  Thus $y\in V\subsetneqq B$ and $U\subseteq V$ for
\[V=\{z\succeq v\wedge w:v\succ x\text{ and }w\in U\}.\]
If $v,v'\succ x$ and $w,w'\in U$ then, by \eqref{RieszInterpolation}, we have $v'$ with $v,v'\succ v''\succ x$ and, as $U$ is a $\prec$-filter, we also have $w''\in U$ with $w''\prec w,w'$.  By \eqref{Multiplicativity}, $v''\wedge w''\prec v\wedge w,v'\wedge w'$ so $V$ is a $\prec$-filter.  Thus if $U$ is a $\prec$-ultrafilter then $V=U$ and hence $y\in U$, as required.

\item[\eqref{compy}$\Rightarrow$\eqref{ultra}]  Assume \eqref{compy}.  If $U\subsetneqq V\subseteq B$, for some $\prec$-filter $V$, then we can take $y\in V\setminus U$.  As $V$ is $\prec$-coinitial, we have $x\in V$ with $x\prec y$.  By \eqref{compy}, we have $w\perp x$, for some $w\in U\subseteq V$.  As $V$ is a $\prec$-filter, $0=w\wedge x\in V$ so $V=B$.

\item[\eqref{prime}$\Rightarrow$\eqref{compy}]  Assume \eqref{prime} and take $z\in U$.  By \eqref{RatherBelow}, if $x\prec y\notin U$ then we have $w\perp x$ with $z\preceq w\vee y$ and hence $w\vee y\in U$.  If $w\notin U$ then, as $y\notin U$ and $B\setminus U$ is a $\preceq$-ideal, $w\vee y\notin U$, a contradiction.  Thus $w\in U$.

\item[\eqref{compy}$\Rightarrow$\eqref{prime}]  Assume \eqref{compy} and take $x,y\notin U$.  By \eqref{Decomposition}, for any $z\prec x\vee y$, we have $x'\prec x$ and $y'\prec y$ such that $z=x'\vee y'$.  By \eqref{compy}, we have $u',v'\in U$ with $u'\perp x'$ and $v'\perp y'$.  As $U$ is $\preceq$-directed, we have $w\in U$ with $w\preceq u',v'$ and hence $w\perp x',y'$.  By \eqref{Distributivity}, $w\perp z$ and hence $z\notin U$.  As $z$ was arbitrary, $x\vee y\notin U$.  As $x,y\notin U$ were arbitrary, $B\setminus U$ is a $\preceq$-ideal.\qedhere
\end{itemize}
\end{proof}

\section{Stone Spaces}\label{SS}

\begin{dfn}
The \emph{Stone Space} $\hat{B}$ of $(B,\prec)$ is the set of $\prec$-ultrafilters in $B$ with the topology generated by $(O_x)_{x\in B}$ where
\[O_x=\{U\in\hat{B}:x\in U\}.\]
\end{dfn}

It follows straight from the definitions that $(O_x)_{x\in B}$ not only generates the topology of $\hat{B}$, but is actually a basis for $\hat{B}$.  Indeed, if $U\in O_x\cap O_y$ then $x,y\in U$ so, as $U$ is a $\prec$-filter, we have $z\in U$ with $z\prec x,y$ and hence $U\in O_z\subseteq O_x\cap O_y$.

\begin{prp}\label{xtoBx}
If $X$ is a locally compact Hausdorff space and $B$ is a relatively compact basis of $X$ then $x\mapsto B_x$ is a homeomorphism from $X$ onto $\hat{B}$ where
\begin{equation}\label{bxdef}
B_x=\{O\in B:x\in O\}
\end{equation}
and $\prec$ is the relation $\subset$ on $B$ defined in \eqref{waybelowdef}.
\end{prp}

\begin{proof}
As $X$ is locally compact and $B$ is a basis, every $B_x$ is a $\subset$-filter.  Also $B_x\neq B$, as $\emptyset\in B\setminus B_x$.  For any $U\in\hat{B}$ and $O\in U$ we have $N\in U$ with $\overline{N}\subseteq O$ so
\[\bigcap_{O\in U}O=\bigcap_{O\in U}\overline{O}.\]
As $U$ is a $\subset$-filter, the latter collection of compact sets has the finite interesection property.  Thus we have some $x\in\bigcap U$, i.e. $U\subseteq B_x$ so, by maximality, $U=B_x$.  Thus the range of $x\mapsto B_x$ contains $\hat{B}$.

As $X$ is Hausdorff, $\bigcap B_x=\{x\}$ so $x\mapsto B_x$ is injective.  This also means that if we had $x\in X$ and a filter $U\subseteq B$ properly containing $B_x$ then we would have $\bigcap U=\emptyset$, despite the fact $\bigcap U$ can be represented as an intersection of compact sets with the finite intersection property as above, a contradiction.  So we do indeed have $B_x\in\hat{B}$, for all $x\in X$.

For any $O\in B$, we have $x\in O\ \Leftrightarrow\ O\in B_x$ so the image of any basic open set $O\in B$ is $\{U\in\hat{B}:O\in U\}$ and vice versa.  Thus $x\mapsto B_x$ is a homeomorphism.
\end{proof}

Note that $B$ above is not even required to be a sublattice of open sets.  Thus it is not clear that $\subset$ (equivalently the topology of $X$) can be recovered from $\subseteq$.

\begin{qst}
Could the topology of $X$ above be recovered from $(B,\subseteq)$ instead?
\end{qst}

More interesting is the fact that we have the following converse of \autoref{BasicBasis}.  Thus we have a duality between basic lattices and locally compact Hausdorff spaces.

\begin{thm}\label{Duality}
For any basic lattice $B$, $\hat{B}$ is locally compact Hausdorff and
\begin{align}
\label{capwedge}&O_x\cap O_y=O_{x\wedge y}.\\
\label{cupvee}&O_x\cup O_y=O_{x\vee y}.\\
\label{perpperp}&O_x\perp O_y\ \Leftrightarrow\ x\perp y.\\
\label{subprec}&O_x\subset O_y\ \Leftrightarrow\ x\prec y.\\
\label{Oxclosure}&\overline{O_x}=\bigcap_{x\prec y}O_y.
\end{align}
\end{thm}

\begin{proof}\
\begin{itemize}
\item[\eqref{capwedge}]  Every $\prec$-filter is a $\preceq$-filter, by \autoref{preceqfilters}.

\item[\eqref{cupvee}]  Again, this follows from \autoref{preceqfilters} and \autoref{ultrachars} \eqref{prime}.

\item[\eqref{perpperp}]  This will follow from \eqref{capwedge}, once we show that $O_x=\emptyset\ \Leftrightarrow\ x=0$.  For this, note that if $x\neq 0$ then we have $y\neq0$ with $y\prec x$, by \eqref{Coinitiality}.  Then $\{z\in B:y\prec z\}$ is a $\prec$-filter, by \eqref{RieszInterpolation}, which extends to a $\prec$-ultrafilter $U\in O_x$.
\end{itemize}

Now say we have $C\subseteq B$ and $z\in B$ with $z\not\prec\bigvee F$, for all finite $F\subseteq C$, and let
\[D=\{x\in B:F\subseteq B\text{ is finite and }z\prec x\vee\bigvee F\}.\]
We claim that then
\[\emptyset\neq\bigcap_{x\in D}O_x\subseteq\overline{O_z}\qquad\text{and}\qquad\bigcap_{x\in D}O_x\cap\bigcup_{y\in C}O_y=\emptyset.\]
If $z\prec x\vee\bigvee F,y\vee G$ then, by \eqref{Multiplicativity}, \eqref{Distributivity} and \eqref{LeftAuxiliarity}, $z\prec(x\wedge y)\vee\bigvee(F\cup G)$, so $D$ is a $\preceq$-filter.  By \eqref{vInterpolation}, $D$ is $\prec$-coinitial and hence a $\prec$-filter.  As $z\not\prec\bigvee F$, for all finite $F\subseteq C$, $0\notin D$ so $D$ has some extension $U\in\bigcap_{x\in D}O_x$.  Thus $\emptyset\neq\bigcap_{x\in D}O_x$.

Assume that $U\notin\overline{O_z}$ so we have $y\in U$ with $z\perp y$.  Take $x\in U$ with $x\prec y$.  By \eqref{<Below}, we have $w\perp x$ with $z\prec w\vee y$.  By \eqref{Decomposition}, we have $w'\prec w$ and $y'\prec y$ with $z=w'\vee y'$.  But then $y'=y'\wedge z\preceq y\wedge z=0$ so $z=w'\prec w$ and hence $w\in D$ (taking $F=\emptyset$).  But $w\perp x\in U$ means $w\notin U$, i.e. $U\notin O_w\supseteq\bigcap_{x\in D}O_x$, a contradiction.  Thus $\bigcap_{x\in D}O_x\subseteq\overline{O_z}$.

By \eqref{<Below}, if $x\prec y\in C$ then we have $w\in D\subseteq U$ with $x\perp w$.  Thus $y\notin U$, by \autoref{ultrachars} \eqref{prime}, i.e. $\bigcap_{x\in D}O_x\cap\bigcup_{y\in C}O_y=\emptyset$, proving the claim.

\begin{itemize}
\item[\eqref{subprec}]  The claim with $C=\{y\}$ yields the $\Rightarrow$ part.  Conversely, if $x\prec y$ and $U\in\overline{O_x}$ then $x\Cap z$, for all $z\in U$, and hence $y\in U$, by \autoref{ultrachars} \eqref{compy}.
\item[\eqref{Oxclosure}]  The claim with $C=\emptyset$ yields $\supseteq$ while \eqref{subprec} yields $\subseteq$.
\end{itemize}

By \eqref{subprec}, the claim is saying that $\overline{O_x}$ is compact, for all $x\in B$, and hence $\hat{B}$ is locally compact.  To see that $\hat{B}$ is Hausdorff, take $U,V\in\hat{B}$.  If we had $V\subseteq U$ then we would have $V=U$, by maximality.  So if $U$ and $V$ are distinct then we have $y\in V\setminus U$.  Take $x\in V$ with $x\prec y$.  By \autoref{ultrachars} \eqref{compy}, we have $w\in U$ with $w\perp x$ and hence $U\in O_w\perp O_x\ni V$, by \eqref{perpperp}.
\end{proof}

\begin{cor}\label{basiclat}
Basic lattices characterize $\cap$-closed $\cup$-closed relatively compact bases of (necessarily locally compact) Hausdorff spaces.  More precisely: If $B$ is a $\cap$-closed $\cup$-closed relatively compact basis of Hausdorff $X$ then $(B,\subset)$ is a basic lattice and $\hat{B}$ is homeomorphic to $X$, while if $(B,\prec)$ is a basic lattice then $(O_x)_{x\in B}$ is a $\cap$-closed $\cup$-closed relatively compact basis of $\hat{B}$ isomorphic to $B$.
\end{cor}

Thus we could have equivalently defined basic lattices as those isomorphic to relatively compact open bases closed under $\cap$ and $\cup$ in locally compact Hausdorff spaces.  Likewise, let us call $B$ a \emph{basic join semilattice} if it isomorphic to a relatively compact open basis closed under $\cup$ in a locally compact Hausdorff space.

\begin{qst}
Is there a finite axiomatization of basic join semilattices?\footnote{By \cite[Theorem 2]{Wolk1956}, the basic join semilattices with a maximum (corresponding to the compact case) are precisely those join semilattices satisfying the dual of separtivity in which every maximal Frink ideal is prime.  However, the mention of subsets here, namely ideals, makes this characterization second order rather than first order.}
\end{qst}

Simply replacing \eqref{Multiplicativity} in \autoref{BasicLatticeDfn} with \eqref{RightAuxiliarity} and \eqref{RieszInterpolation}, as in \autoref{nomeets}, does not work, as \eqref{Decomposition} is too strong.  For example, consider the basis $B$ of $X=\mathbb{N}\cup\{\infty\}$, the one-point compactification of $\mathbb{N}$, consisting of all neighbourhoods of $\infty$, finite subsets of $\mathbb{N}$ and $\mathbb{N}$ itself.  Then $\mathbb{N}\subset\{1\}\cup(X\setminus\{1\})$ but the only we way we could have $O\subset\{1\}$ and $N\subset X\setminus\{1\}$ with $\mathbb{N}=O\cup N$ is if $N=\mathbb{N}\setminus\{1\}$, even though $\mathbb{N}\setminus\{1\}\notin B$.

We could further replace \eqref{Decomposition} with \eqref{vInterpolation}, but this is too weak.  Indeed, note that every finite basic join semilattice is isomorphic to the lattice of all subsets $\mathcal{P}(X)$ of some finite $X$ (where $\prec$ coincides with $\preceq$), and there are plenty of other finite lattices satisfying these axioms, e.g. the diamond lattice $D_3$ with minimum $0$, maximum $1$ and three incomparable elements in between.  In fact, when $\prec$ is reflexive and hence coincides with $\preceq$, we can always just take $x'=x$ and $y'=y$ in \eqref{vInterpolation}.

On the other hand, we know there is no finite axiomatization of \emph{basic meet semilattices} (defined like basic join semilattices but with $\cap$ replacing $\cup$).  In fact, basic meet semilattices do not even form an elementary class.  To see this note that, for every $n\in\mathbb{N}$, $D_n$ is a basic meet semilattice (representing a basis of the discrete space with $n$ points), but an ultraproduct of $(D_n)$ is not.  Indeed, such an ultraproduct is isomorphic to $D_\kappa$ for some (uncountable) infinite $\kappa$.  If such a lattice represented a basis of relatively open sets closed under $\cap$ in locally compact Hausdorff $X$, each element of $D_\kappa\setminus\{0,1\}$ would represent an isolated point, while $1$ would represent $X$ which, being relatively compact by definition, must actually be compact.  As infinite collections of isolated points are not compact we must have some other point $x\in X$.  As $X$ is Hausdorff and $B$ is a basis, there must be some proper open set containing $x$ represented in $D_\kappa$, a contradiction.  We will return to this problem in \autoref{BS}.

Lastly, let us note that $\subset$ is reflexive precisely when the basis elements are not just open but also compact.  Thus we get a duality between $0$-dimensional locally compact Hausdorff spaces and generalized Boolean algebras.  And a $\cup$-closed basis has a maximum precisely when $X$ is compact, so in this case we recover the classical Stone duality between $0$-dimensional compact Hausdorff spaces and Boolean algebras.

\section{Interpolators}\label{I}

If we really hope to do topology in a first order way, we also need a first order analog of continuous maps.  For this, we introduce interpolators.

\begin{dfn}
For basic lattices $(B,\prec)$ and $(C,<)$, call $\sqsubset\ \subseteq B\times C$ an \emph{interpolator} if it satisfies \eqref{Minimum} and all the axioms from \eqref{Cofinality} to \eqref{Decomposition}.
\end{dfn}

Actually, \eqref{Interpolation} for $\sqsubset$ above really becomes two axioms
\begin{align}
\label{<Interpolation}\tag{$<$-Interpolation}x\sqsubset y\qquad&\Rightarrow\qquad\exists z\in C\ (x\sqsubset z<y).\\
\label{precInterpolation}\tag{$\prec$-Interpolation}x\sqsubset y\qquad&\Rightarrow\qquad\exists z\in B\ (x\prec z\sqsubset y).
\end{align}

\begin{prp}
If $\sqsubset$ is an interpolator then we also have
\begin{align}
\label{leqAuxiliarity}\tag{$\leq$-Auxiliarity}x\sqsubset z\leq y\qquad&\Rightarrow\qquad x\prec y.\\
\label{preceqAuxiliarity}\tag{$\preceq$-Auxiliarity}x\preceq z\sqsubset y\qquad&\Rightarrow\qquad x\prec y.
\end{align}
\end{prp}

\begin{proof}
If $x\sqsubset z\leq y$ then $0\sqsubset y$, by \eqref{Minimum}, so $x=0\vee x\sqsubset y\vee z=y$, by \eqref{Additivity}.  Thus if $x\preceq z\sqsubset y$ then we have $w\sqsupset x$, by \eqref{Cofinality}, so $x=x\wedge z\sqsubset w\wedge y\preceq y$, by \eqref{Multiplicativity}, and hence $x\sqsubset y$, by \eqref{leqAuxiliarity}.
\end{proof}

We define composition of relations in the usual way, namely by
\[x\mathbin{\sqsubset}\circ\mathbin{\sqsubset'}y\qquad\Leftrightarrow\qquad\exists z\ (x\sqsubset z\sqsubset'y).\]
It is routine to verify that a composition of interpolators is again an interpolator.  It is also immediate from the definitions that if $(B,\prec)$ is a basic lattice then $\prec$ is an interpolator from $B$ to itself with $\prec\ =\mathbin{\prec}\circ\mathbin{\prec}$, as $\prec$ is transitive and satisfies \eqref{Interpolation}.  So taking interpolators as morphisms turns the class of basic lattices into a category, which we denote by $\mathbf{BasLat}$.  We also let $\mathbf{LocHaus}$ denote the category of locally compact Hausdorff topological spaces with continuous maps as morphisms.

\begin{thm}\label{categoryequiv}
$\mathbf{BasLat}$ and $\mathbf{LocHaus}$ are equivalent categories.
\end{thm}

\begin{proof}
Assume $\sqsubset$ is an interpolator from $(B,\prec)$ to $(C,<)$.  For $U\in\hat{B}$, define
\[U^\sqsubset=\{y\in C:\exists x\in U(x\sqsubset y)\}.\]
We claim that $U^\sqsubset\in\hat{C}$.  By \eqref{leqAuxiliarity}, $U^\sqsubset$ is $\geq$-closed.  By \eqref{Multiplicativity} and \eqref{<Interpolation}, $U^\sqsubset$ is $<$-directed.  And by \eqref{Decomposition}, $U^\sqsubset$ satisfies \autoref{ultrachars} \eqref{prime}, proving the claim.

Let $f(U)=U^<$ so $f:\hat{B}\rightarrow\hat{C}$.  As $f^{-1}[O_y]=\bigcup\limits_{x\sqsubset y}O_x$, $f$ is continuous.  We claim
\[x\sqsubset y\qquad\Leftrightarrow\qquad f[\overline{O_x}]\subseteq O_y.\]
If $x\sqsubset y$ then \eqref{precInterpolation} yields $z\in B$ with $x\prec z\sqsubset y$.  Then \eqref{subprec} yields
\[f[\overline{O_x}]\subseteq f[O_z]\subseteq O_y.\]
Conversely, assume $x\not\sqsubset y$ and let
\[D=\{z\in B:\exists w\sqsubset y\ (x\prec w\vee z)\}.\]
By $\prec$-Auxiliarity, if $0\in D$ then $x\sqsubset y$, a contradiction, so $0\notin D$.  Again using $\prec$-Auxiliarity and arguing as in the proof of \autoref{Duality}, $D$ can be extended to $U\in\overline{O_x}\setminus\bigcup\limits_{w\sqsubset y}O_w$ so $f(U)\notin O_y$ and hence $f[\overline{O_x}]\not\subseteq O_y$, proving the claim.

On the other hand, if $f:X\rightarrow Y$ is a continuous map between locally compact Hausdorff spaces with relatively compact bases $B$ and $C$ respectively, it is routine to verify that we get an interpolator $\sqsubset$ defined by
\[O\sqsubset N\qquad\Leftrightarrow\qquad f[\overline{O}]\subseteq N.\]
Then we immediately see that $C_{f(x)}=B_x^\sqsubset$, for all $x\in X$.
\end{proof}

\section{Basic Semilattices}\label{BS}

Our next goal is to characterize basic meet semilattices by omitting types.  By a `type' we mean a collection of first order formulas and by `omit' we mean that there are no elements which satisfy the entire type (see \cite[Ch 4]{Marker2002}).  Specifically, consider the types \eqref{phin} and \eqref{psin} where
\begin{align}
\label{phin}\tag{$\phi_n$}\phi_n(x,y)\ \ \Leftrightarrow\ \ x\prec y\text{ and }&\forall w_1,\ldots,w_n\prec y\ \forall v_1\prec w_1\ldots\forall v_n\prec w_n\\
\nonumber&\exists x'\neq0\ (x\succ x'\perp v_1,\ldots,v_n).\\
\label{psin}\tag{$\psi_n$}\psi_n(x,y,z)\ \ \Leftrightarrow\ \ x\prec y\text{ and }&\forall y'\prec y\ \forall w_1,\ldots,w_n\perp x\ \forall v_1\prec w_1\ldots\forall v_n\prec w_n\\
\nonumber&\exists z'\neq0\ (z\succ z'\perp y',v_1,\ldots,v_n).
\end{align}
As we shall soon see, omitting \eqref{phin} corresponds to \eqref{Interpolation}, while omitting \eqref{psin} corresponds to \eqref{<Below}.  Also let \eqref{thetan} be the sentences given by
\[\label{thetan}\tag{$\theta_n$}\theta_n\quad\Leftrightarrow\quad\forall x,y\ (\exists w_1\ldots,w_n\prec y\ \nexists v\neq0\ (x\succ v\perp w_1,\ldots,w_n)\ \Rightarrow\ x\prec y).\]

\begin{dfn}
We call $(B,\prec)$ a \emph{basic semilattice} if $(B,\preceq)$ is a meet semilattice and $(B,\prec)$ satisfies \eqref{Minimum}, \eqref{Transitivity}, \eqref{Coinitiality}, \eqref{Multiplicativity} and \eqref{thetan}, while omitting \eqref{phin} and \eqref{psin}.
\end{dfn}

Let $\mathcal{P}(B)$, $\mathcal{F}(B)$ and $\mathcal{S}(B)$ be the arbitrary, finite and singleton subsets of $B$,
\begin{align*}
\mathcal{P}(B)&=\{C:C\subseteq B\}.\\
\mathcal{F}(B)&=\{C\subseteq B:|C|<\infty\}.\\
\mathcal{S}(B)&=\{\{x\}:x\in B\}.
\end{align*}
Consider the relations defined on $\mathcal{P}(B)$ by
\begin{align*}
C&\prec D\quad&\Leftrightarrow\quad&& C&\subseteq D^\succ&\Leftrightarrow\qquad&\forall x\in C\ \exists y\in D\ (x\prec y).\\
C&\precsim D\quad&\Leftrightarrow\quad&& C^\succ&\subseteq D^\Cap\cup\{0\}&\Leftrightarrow\qquad&\forall x\in C\ \nexists y\perp D\ (x\succ y\neq0).
\end{align*}
By \eqref{Coinitiality}, $\C$ is reflexive except at $0$ so $A^\succ\subseteq A^{\succ\Cap}\cup\{0\}=A^\Cap\cup\{0\}$ and hence, by defintion, $A\precsim A\precsim A^\succ$, i.e. $\precsim$ is reflexive too.  Also define
\begin{equation}\label{CwedgeDdef}
C\wedge D=\{x\wedge y:x\in C\text{ and }y\in D\}.
\end{equation}

\begin{prp}\label{precprops}
On $\mathcal{P}(B)$, $\prec$ and $\precsim$ are transitive, $\bigcup$-additive, $\wedge$-multiplicative and
\[C\preceq D\quad\Rightarrow\quad C\precsim D\quad\Leftrightarrow\quad C\underset{\displaystyle\sim}{\preceq}D.\]
\end{prp}

\begin{proof}\
\begin{itemize}
\item[($\prec$-transitivity)] If $C\prec D\prec E$ then $C\subseteq D^\succ\subseteq E^{\succ\succ}\subseteq E^\succ$, as $\prec$ satisfies \eqref{Transitivity} on $B$, i.e. $C\prec E$ so $\prec$ is transitive on $\mathcal{P}(B)$ too.

\item[($\bigcup$-additivity)] As $(\bigcup\mathcal{P})^\succ=\bigcup_{A\in P}A^\succ$ and $(\bigcup\mathcal{P})^\Cap=\bigcup_{A\in P}A^\Cap$, for any $\mathcal{P}\subseteq\mathcal{P}(B)$, the $\bigcup$-additivity of $\prec$ and $\precsim$ follows immediately from the $\bigcup$-additivity of $\subseteq$.

\item[($\prec$-$\wedge$-multiplicativity)] If $C\prec C'$ and $D\prec D'$ then $C\wedge D\subseteq C^\succeq\cap D^\succeq\subseteq C'^\succ\cap D'^\succ\subseteq(C'\wedge D')^\succ$, by \eqref{LeftAuxiliarity} and \eqref{Multiplicativity}, so $\prec$ is $\wedge$-multiplicative.

\item[($\precsim$-$\wedge$-multiplicativity)] If $C\precsim C'$ and $D\precsim D'$ then, for every $z\in(C\wedge D)^\succ\backslash\{0\}\subseteq C^\succ\backslash\{0\}$, we have $x'\in C'$ with $z\C x'$.  So $z\wedge x'\in D^\succ\backslash\{0\}$ and hence we have $y'\in D'$ with $z\wedge x'\C y'$, i.e. $z\wedge x'\wedge y'\neq0$ so $z\C x'\wedge y'$ and hence $z\in(C'\wedge D')^\Cap$.

\item[($C\preceq D\ \Rightarrow\ C\precsim D$)] Note $C\subseteq D^\succ$ implies $C^\succ\subseteq D^{\succ\succ}\subseteq D^\succ\subseteq D^\Cap\cup\{0\}$.

\item[($C\precsim D\ \Leftrightarrow\ C\underset{\displaystyle\sim}{\preceq}D$)] If $C\precsim D$ then, for all $x\in C^\succeq\backslash\{0\}$, \eqref{Coinitiality} yields non-zero $x'\prec x$\vspace{-4pt}
so $x'\in C^{\succeq\succ}\backslash\{0\}\subseteq C^\succ\backslash\{0\}\subseteq D^\Cap$ and hence $x\in D^{\Cap\prec}=D^\Cap$.  Thus $C^\succeq\subseteq D^\Cap\cup\{0\}$, i.e. $C\underset{\displaystyle\sim}{\preceq}D$.  The converse is immediate from $C^\succ\subseteq C^\succeq$.

\item[($\precsim$-transitivity)] If $C\precsim D\precsim E$ then, for every $x\in C^\succ\backslash\{0\}$, we have $y\in D$ with $x\C y$, i.e. $0\neq x\wedge y\preceq y$ so $x\wedge y\in D^\succeq\backslash\{0\}\subseteq E^\Cap$ and hence $x\in E^{\Cap\prec}=E^\Cap$.  Thus $C^\succ\subseteq E^\Cap\cup\{0\}$, i.e. $C\precsim E$ so $\precsim$ is transitive too.\qedhere
\end{itemize}
\end{proof}

Define $\pp$ on $\mathcal{P}(B)$ (representing \eqref{waybelowdef} \textendash\, see \eqref{CssD}) by
\[C\pp D\qquad\Leftrightarrow\qquad\exists F\in\mathcal{F}(B)\ (C\precsim F\prec D).\]

\begin{prp}\label{FprecA}
$\pp$ is transitive, $\cup$-additive, $\wedge$-multiplicative and, for $F\in\mathcal{F}(B)$,
\begin{align}
\label{FprecD}F\prec D\quad&\Rightarrow\quad F\pp D.\\
\label{CppD}\exists G\in\mathcal{F}(B)\ (C\pp G\prec D)\quad&\Leftrightarrow\quad C\pp D.\\
\label{CprecsimEppD}C\precsim E\pp D\quad&\Rightarrow\quad C\pp D\quad\Rightarrow\quad C\precsim D.
\end{align}
\end{prp}

\begin{proof}\
\begin{itemize}
\item[\eqref{FprecD}]  If $F\prec A$ then $F\precsim F\prec A$ so $F\pp A$.

\item[\eqref{CppD}] If $C\pp D$ then $C\precsim F\prec D$, for some $F\in\mathcal{F}(B)$.  Thus, for all $x\in F$, we have $y_x\in D$ with $x\prec y_x$.  As $B$ omits \eqref{phin}, we have $V_x,W_x\in\mathcal{F}(B)$ such that $\{x\}\precsim V_x\prec W_x\prec\{y_x\}$.  Setting $G=\bigcup_{x\in F}W_x$ and $H=\bigcup_{x\in F}V_x$, $\cup$-additivity yields $C\precsim F\precsim H\prec G\prec\{y_x:x\in F\}\subseteq D$.  As $\precsim$ is transitive, $C\precsim H\prec G\prec D$ so $C\pp G\prec D$.

If $C\pp G\prec D$ then we have $F\in\mathcal{F}(B)$ with $C\precsim F\prec G\prec D$ so the transitivity of $\prec$ yields $C\precsim F\prec D$, i.e. $C\pp D$.

\item[\eqref{CprecsimEppD}] If $C\precsim E\pp D$ then we have $F\in\mathcal{F}(B)$ with $C\precsim E\precsim F\prec D$ so the transitivity of $\precsim$ yields $C\precsim F\prec D$, i.e. $C\pp D$.

If $C\pp D$ then we have $F\in\mathcal{F}(B)$ with $C\precsim F\prec D$ so, as $\prec$ is stronger than $\precsim$, $C\precsim F\precsim D$ and hence $C\precsim D$, again by the transitivity of $\precsim$.
\end{itemize}

As $\prec$ and $\precsim$ are $\cup$-additive and $\wedge$-multiplicative, so is $\pp$.

If $C\pp D\pp E$ then $C\precsim F\prec D\precsim G\prec E$, for $F,G\in\mathcal{F}(B)$.  As $F\prec D$ implies $F\precsim D$ and $\precsim$ is transitive, $C\precsim G\prec E$, i.e. $C\pp E$ so $\pp$ is transitive too.
\end{proof}

Define the \emph{saturation} of any $A\subseteq B$ by
\[A^\cup={\displaystyle\bigcup_{C\pp A}}C=\{y\in B:\{y\}\pp A\}.\]

\begin{prp}\label{cupcup}
For all $A,C\subseteq B$ and $F\in\mathcal{F}(B)$,
\begin{align}
\label{FppA}F\subseteq A^\cup\quad&\Leftrightarrow\quad F\pp A.\\
\label{CppA}C\pp A^\cup\quad&\Leftrightarrow\quad C\pp A\quad\Leftrightarrow\quad C^\cup\pp A.
\end{align}
\begin{equation}\label{Acup=}
A^\succ\subseteq A^\cup=A^{\cup\succeq}=A^{\succeq\cup}=A^{\succ\cup}=A^{\cup\cup}=\bigcup_{G\in\mathcal{F}(A)}G^\cup\precsim A.
\end{equation}
\end{prp}

\begin{proof}\
\begin{itemize}
\item[\eqref{FppA}]  If $F\pp A$ then $F\subseteq A^\cup$ by definition.  Conversely, if $F\subseteq A^\cup$ then, for each $y\in F$, we have some $G_y\in\mathcal{F}(B)$ with $\{y\}\precsim G_y\prec A$.  Thus $\cup$-additivity yields $F\precsim\bigcup_{y\in F}G_y\prec A$ and hence $F\pp A$.

\item[($A^\cup\precsim A$)]  If $\{x\}\pp A$ then $\{x\}\precsim A$ and hence $A^\cup\precsim A$ by the $\bigcup$-additivity of $\precsim$.

\item[($A^\cup=\bigcup_{F\in\mathcal{F}(A)}F^\cup$)] If $\{y\}\pp A$ then $\{y\}\precsim F\prec A$, for some $F\in\mathcal{F}(B)$, so $F\prec G$, for some $G\in\mathcal{F}(A)$, and hence $\{y\}\precsim F\prec G$, i.e. $\{y\}\pp G$.

\item[($A^\succ\subseteq A^\cup$)]  If $x\prec y\in A$ then $\{x\}\precsim\{x\}\prec A$ so $\{x\}\pp A$.

\item[($A^\cup=A^{\cup\succeq}$)]  If $x\preceq y\in A^\cup$ then $\{x\}\precsim\{y\}\precsim G\prec A$, for some $G\in\mathcal{F}(B)$.  As $\precsim$ is transitive, $\{x\}\precsim G\prec A$ and hence $\{x\}\pp A$.

\item[($A^\cup=A^{\succeq\cup}$)]  By \eqref{LeftAuxiliarity}, $G\prec A$ iff $G\prec A^\succeq$.

\item[($A^{\cup\cup}=A^\cup=A^{\succ\cup}$)]  By \eqref{CppD}, if $y\in A^\cup$ then $\{y\}\pp G\subseteq A^\succ$, for some $G\in\mathcal{F}(B)$, showing $A^\cup\subseteq A^{\succ\cup}\subseteq A^{\cup\cup}$.  Conversely, if $y\in A^{\cup\cup}$ then $y\in G^\cup$, for some $G\in\mathcal{F}(A^\cup)$, so $\{y\}\pp G\pp A$ and hence $\{y\}\pp A$, showing $A^{\cup\cup}\subseteq A^\cup$.

\item[\eqref{CppA}]  If $C\pp A^\cup$ then we have $F\in\mathcal{F}(B)$ with $C\precsim F\prec A^\cup$.  Therefore $F\subseteq A^{\cup\succ}\subseteq A^{\cup\succeq}=A^\cup$ so $C\precsim F\pp A$ and hence $C\pp A$, by \eqref{CprecsimEppD}.

Conversely, if $C\pp A$ then we have $G\in\mathcal{F}(B)$ with $C\pp G\prec A$, by \eqref{CppD}.  Thus $G\subseteq A^\succ\subseteq A^\cup$ so $C\pp A^\cup$.

If $C^\cup\pp A$ then $C\precsim C^\succ\subseteq C^\cup\pp A$ so $C\pp A$ too.

Conversely, if $C\pp A$ then $C^\cup\precsim C\pp A$ so $C^\cup\pp A$.\qedhere
\end{itemize}
\end{proof}

In particular, $A\mapsto A^\cup$ behaves much like a closure operator on $\mathcal{P}(B)$, being idempotent ($A=A^{\cup\cup}$), increasing ($A\subseteq B\Rightarrow A^\cup\subseteq B^\cup$) and even finitary ($A^\cup=\bigcup_{G\in\mathcal{F}(A)}G^\cup$).  However, it is usually not extensive, i.e. we can have $A\nsubseteq A^\cup$.

In the next result we use some standard terminology from frame and domain theory (see e.g. \cite{GierzHofmannKeimelLawsonMisloveScott2003}, \cite{PicadoPultr2012} or \cite{Goubault2013}).  Specifically, by a \emph{frame} we mean a complete lattice $L$ where finite meets distribute over arbitrary joins.  For $x,y\in L$, we say $x$ is \emph{way-below} $y$ if $x\in Z^\succeq$ whenever $y\preceq\bigvee Z$ for $\succeq$-directed $Z$.  And we say $L$ is \emph{continuous} if every $x\in L$ is the join of those elements way-below $x$.

Denote the saturated subsets generated by $\mathcal{P}\subseteq\mathcal{P}(B)$ by
\[\mathcal{P}^\cup=\{A^\cup:A\in\mathcal{P}\}.\]
Just as the saturated subsets in formal topology yield frames (see \cite{CirauloMaiettiSambin2013} Definition 4.1), saturated subsets of basic semilattices yield continuous frames.

\begin{thm}\label{contframe}
$(\mathcal{P}(B)^\cup,\subseteq)$ is a continuous frame with way-below relation $\pp$ and 
\begin{align}
\label{supP}\bigvee_{A\in\mathcal{P}}A^\cup&=(\bigcup\mathcal{P})^\cup&(\text{taking $\bigvee$ in }\mathcal{P}(B)^\cup).\\
\label{CcapD}C^\cup\wedge D^\cup&=(C\wedge D)^\cup=C^\cup\cap D^\cup&(\text{taking $\wedge$ from \eqref{CwedgeDdef}}).
\end{align}
So meets in $\mathcal{P}(B)^\cup$ coincide with both $\cap$ and $\wedge$ on $\mathcal{P}(B)$.
\end{thm}

\begin{proof}\
\begin{itemize}
\item[\eqref{supP}] For all $A\in\mathcal{P}$, $A\subseteq\bigcup\mathcal{P}$ so $A^\cup\subseteq(\bigcup\mathcal{P})^\cup$.  Conversely, if $O\in\mathcal{P}(B)^\cup$ and $A^\cup\subseteq O$, for all $A\in\mathcal{P}$, then, by \autoref{cupcup},
\[(\bigcup\mathcal{P})^\cup=(\bigcup\mathcal{P})^{\succ\cup}=(\bigcup_{A\in\mathcal{P}}A^\succ)^\cup\subseteq(\bigcup_{A\in\mathcal{P}}A^\cup)^\cup\subseteq O^\cup=O.\]

\item[\eqref{CcapD}] $C^\cup\cap D^\cup\subseteq C^\cup\wedge D^\cup$ is immediate and the $\wedge$-multiplicativity of $\pp$ yields $C^\cup\cap D^\cup\subseteq(C\wedge D)^\cup$.  Conversely, \autoref{cupcup} yields
\begin{align*}
(C\wedge D)^\cup&\subseteq C^{\succeq\cup}\cap D^{\succeq\cup}=C^\cup\cap D^\cup.\\
C^\cup\wedge D^\cup&\subseteq C^{\cup\succeq}\cap D^{\cup\succeq}=C^\cup\cap D^\cup.
\end{align*}

\item[(Distributivity)] Take $A\subseteq B$ and $\mathcal{P}\subseteq\mathcal{P}(B)$.  By \eqref{CcapD},
\[A^\cup\wedge\bigvee\mathcal{P}^\cup=(A\wedge\bigcup\mathcal{P})^\cup=(\bigcup_{C\in\mathcal{P}}(A\wedge C))^\cup=\bigvee_{C\in\mathcal{P}}(A\wedge C)^\cup=\bigvee_{C\in\mathcal{P}}(A^\cup\wedge C^\cup).\]

\item[(Way-Below)] Take $C\subseteq B$ so $C^\cup=C^{\succ\cup}=\bigcup_{F\in\mathcal{F}(C^\succ)}F^\cup$ and this latter union is directed.  Thus if $D\subseteq B$ and $D^\cup$ is way-below $C^\cup$ in $\mathcal{P}(B)^\cup$ then by definition $D^\cup\subseteq F^\cup$, for some $F\in\mathcal{F}(C^\succ)$.  Thus $D\precsim D^\succ\subseteq D^\cup\subseteq F^\cup\precsim F\prec C$ so $D\precsim F\prec C$, i.e. $D\pp C$.

By \eqref{FprecD} and \eqref{CppA}, $F^\cup\pp C^\cup$, for all $F\in\mathcal{F}(C^\succ)$, so the continuity of $\mathcal{P}(B)^\cup$ will follow from the converse.  For this, assume $D\pp C\subseteq\bigvee_{A\in\mathcal{P}}A^\cup$, for some $C,D\subseteq B$ and $\mathcal{P}\subseteq\mathcal{P}(B)$.  So we have $F\in\mathcal{F}(B)$ with
\[D\precsim F\prec\bigvee_{A\in\mathcal{P}}A^\cup=(\bigcup\mathcal{P})^\cup=\bigcup_{G\in\mathcal{F}(\bigcup\mathcal{P})}G^\cup.\]
As this last union is directed, we have $G\in\mathcal{F}(\bigcup\mathcal{P})$ with $D\precsim F\prec G^\cup$, i.e. $D\pp G^\cup$ so $D\pp G$, by \eqref{CppA}, and hence $D\subseteq G^\cup$.  Taking finite $\mathcal{G}\subseteq\mathcal{P}$ with $G\subseteq\bigcup\mathcal{G}$ yields $D\subseteq(\bigcup\mathcal{G})^\cup=\bigvee_{A\in\mathcal{G}}A^\cup$.  As $\mathcal{P}$ was arbitrary, this shows that $D$ is way-below $C$ in $\mathcal{P}(B)^\cup$, as long as $C$ and $D$ are in $\mathcal{P}(B)^\cup$.\qedhere
\end{itemize}
\end{proof}

Recall that we denote the singleton subsets of $B$ by $\mathcal{S}(B)=\{\{b\}:b\in B\}$.  We call a subset $S$ of a lattice $L$ \emph{$\bigvee$-dense} (\emph{$\vee$-dense}) if every element of $L$ is a (finite) join of elements in $S$.

\begin{cor}\label{semicor}\
\begin{enumerate}
\item\label{MSSL} $(\mathcal{S}(B)^\cup,\subseteq)$ is a $\vee$-dense meet subsemilattice of $(\mathcal{F}(B)^\cup,\subseteq)$.
\item\label{SL} $(\mathcal{F}(B)^\cup,\subseteq)$ is a $\bigvee$-dense sublattice of $(\mathcal{P}(B)^\cup,\subseteq)$.
\item\label{ISO} $(\mathcal{S}(B)^\cup,\pp)$ is isomorphic to $(B,\prec)$.
\item\label{FBL} $(\mathcal{F}(B)^\cup,\pp)$ is a basic lattice.
\end{enumerate}
\end{cor}

\begin{proof}\
\begin{enumerate}
\item[\eqref{MSSL}] By \eqref{CcapD}, the meet of $\{x\}^\cup$ and $\{y\}^\cup$ in $(\mathcal{P}(B)^\cup,\subseteq)$ is $\{x\wedge y\}^\cup$.  By \eqref{supP}, $F^\cup=(\bigcup_{x\in F}\{x\})^\cup=\bigvee_{x\in F}\{x\}^\cup$, for all $F\in\mathcal{F}(B)$, so $\vee$-density follows.

\item[\eqref{SL}] By \eqref{supP}, the join of $F^\cup$ and $G^\cup$ in $(\mathcal{P}(B)^\cup,\subseteq)$ is $(F\cup G)^\cup$.  Thus, as $\mathcal{F}(B)$ is $\cup$-closed, $(\mathcal{F}(B)^\cup,\subseteq)$ is a sublattice of $(\mathcal{P}(B)^\cup,\subseteq)$.  Again by \eqref{supP} we obtain $\bigvee$-density.

\item[\eqref{ISO}] If $x\prec y$ then $\{x\}\pp\{y\}$, by \eqref{FprecD}, and hence $\{x\}^\cup\pp\{y\}^\cup$, by \eqref{CppA}.  As $B$ satisfies \eqref{thetan}, the converse also holds.

\item[\eqref{FBL}] We verify a sufficient collection of axioms from \autoref{BL}.
\begin{itemize}
\item[\eqref{Multiplicativity}] See \autoref{FprecA}.

\item[\eqref{Distributivity}] By \autoref{contframe}, $(\mathcal{P}(B)^\cup,\subseteq)$ is distributive, thus so is any sublattice.

\item[\eqref{vInterpolation}] This holds for any $\bigvee$-dense sublattice of a continuous lattice, so again this follows from \autoref{contframe}.

\item[\eqref{Additivity}] Likewise, this holds for any join subsemilattice of a continuous lattice.

\item[\eqref{<Below}] As $B$ omits \eqref{psin}, for any $x,y,z\in B$ with $x\prec y$, we have $y'\prec y$ and $V,W\in\mathcal{F}(B)$ with $\{z\}\precsim\{y'\}\cup V$ and $V\prec W\perp x$.  We need to extend this to $\mathcal{F}(B)$.  So take $X,Y,Z\in\mathcal{F}(B)$ with $X\pp Y$, which means we have $F\in\mathcal{F}(B)$ with $X\precsim F\prec Y$.  For each $x\in F$, we have $y_x\in Y$ with $x\prec y_x$ and thus, for each $z\in Z$, we have $y'_{x,z}\prec y_x$ and $V_{x,z},W_{x,z}\in\mathcal{F}(B)$ with $\{z\}\precsim\{y'_{x,z}\}\cup V_{x,z}$ and $V_{x,z}\prec W_{x,z}\perp x$.  Let
\begin{align*}
U&=\{y_{x,z}':x\in F,z\in H\}.\\
V&=\bigwedge_{x\in F}\bigcup_{z\in H}V_{x,z}.\\
W&=\bigwedge_{x\in F}\bigcup_{z\in H}W_{x,z}.
\end{align*}
By \autoref{precprops}, $U\prec Y$, $V\prec W\perp F\succsim X$ and $Z\precsim U\cup V\prec Y\cup W$ so $W\perp X$ and $Z\pp Y\cup W$.  Thus $\mathcal{F}(B)^\cup$ satisfies \eqref{<Below}, by \eqref{CppA} and \eqref{supP}.\qedhere
\end{itemize}
\end{enumerate}
\end{proof}

\begin{prp}\label{basicsemi}
If $X$ is a locally compact Hausdorff space with $\cap$-closed basis $B$ of relatively compact open sets then $(B,\subset)$ is a basic semilattice,
\begin{align}
\label{CprecsimD}C\subsetsim D\qquad&\Leftrightarrow\qquad\bigcup C\subseteq\overline{\bigcup D},\\
\label{CssD}C\ssubset D\qquad&\Leftrightarrow\qquad\bigcup C\subset\bigcup D,
\end{align}
for all $C,D\subseteq B$, and $\mathcal{P}^\cup=\{O\in B:O\subset\bigcup\mathcal{P}\}$, for all $\mathcal{P}\subseteq\mathcal{P}(B)$.
\end{prp}

\begin{proof}\
\begin{itemize}
\item[\eqref{CprecsimD}]  If $\bigcup C\subseteq\overline{\bigcup D}$ and $\emptyset\neq O\in C^\supset$ then $O\subseteq C\subseteq\overline{\bigcup D}$ so $O\cap N\neq\emptyset$, for some $N\in D$, i.e. $O\in D^\Cap$ so $C\subsetsim D$.  Conversely, if $\bigcup C\nsubseteq\overline{\bigcup D}$ then $\emptyset\neq O\setminus\overline{\bigcup D}$, for some $O\in C$.  As $B$ is a basis, we have some $N\in B$ with $\emptyset\neq N\subset O\setminus\overline{\bigcup D}$, i.e. $N\in C^\supset\backslash D^\Cap$ so $C\not\subsetsim D$.
\item[\eqref{CssD}]  If $C\subsetsim F\subset D$, for some $F\in\mathcal{F}(B)$, then $\bigcup C\subseteq\overline{\bigcup F}$ so
\[\overline{\bigcup C}\subseteq\overline{\overline{\bigcup F}}=\overline{\bigcup F}=\bigcup_{O\in F}\overline{O}\subseteq\bigcup D.\]
As $\overline{\bigcup F}$ is compact, $\overline{\bigcup C}$ is also compact so $\bigcup C\subset\bigcup D$.  Conversely, if $\bigcup C\subset\bigcup D$ then, for each $x\in\overline{\bigcup C}$, we have $O_x\in B$ with $x\in O_x\subset N$, for some $N\in D$.  As $\overline{\bigcup C}$ is compact, $(O_x)$ has a finite subcover $F$ so $\bigcup C\subseteq\overline{\bigcup C}\subseteq\bigcup F\subseteq\overline{\bigcup F}$, i.e. $C\subsetsim F\subset D$.
\end{itemize}
To see that $B$ omits \eqref{phin}, take $O,N\in B$ with $O\subset N$.  For each $x\in\overline{O}\subseteq N$, we have $V_x,W_x\in B$ with $x\in V_x\subset W_x\subset N$.  As $\overline{O}$ is compact, we have some subcover of size $n<\infty$, showing that $\phi_n(O,N)$ fails.  Similar compactness arguments show that $B$ omits \eqref{psin} and satisfies \eqref{thetan}.  Also \eqref{Coinitiality} and \eqref{Multiplicativity} are immediate so $B$ is a basic semilattice.
\end{proof}

By \eqref{CssD}, $\mathcal{P}\mapsto\bigcup\mathcal{P}$ is an isomorphism from $(\mathcal{P}(B)^\cup,\ssubset)$ to $(\mathcal{O}(X),\subset)$, with inverse $O\mapsto\{N\in B:N\subset O\}$.  If $B$ is also $\cup$-closed, then the sets of the form $\{N\in B:N\subset O\}$ are precisely the $\subset$-ideals of $B$.  In other words, when $(B,\prec)$ is a basic lattice, $\mathcal{P}(B)^\cup$ consists precisely of the $\prec$-ideals of $B$, by \autoref{basiclat}.  Thus $\mathcal{P}(B)^\cup$ can be seen as a generalization of the `rounded ideal completion' of $B$ (see \cite[Proposition 5.1.33]{Goubault2013}) from basic lattices to basic semilattices.

\begin{cor}\label{semichar}
Basic semilattices characterize $\cap$-closed relatively compact bases of (necessarily locally compact) Hausdorff spaces.  More precisely, if $B$ is a relatively compact basis of Hausdorff $X$ then $(B,\subset)$ is a basic semilattice and $\hat{B}$ is homeomorphic to $X$, while if $(B,\prec)$ is a basic semilattice then $\hat{B}$ has a relatively compact basis isomorphic to $B$.
\end{cor}

\begin{proof}
If $B$ is a relatively compact basis of Hausdorff $X$ then $(B,\subset)$ is a basic semilattice, by \autoref{basicsemi}, and $\hat{B}$ is homeomorphic to $X$, by \autoref{xtoBx}.

If $(B,\prec)$ is a basic semilattice then $B$ is isomorphic to a $\vee$-dense meet subsemilattice of the basic lattice $(\mathcal{F}(B)^\cup,\subset)$, by \autoref{semicor}.  Thus $B$ is isomorphic to a $\cap$-closed relatively compact basis of a (locally compact) Hausdorff space, by \autoref{basiclat}.  By \autoref{xtoBx}, this space is homeomorphic to $\hat{B}$.
\end{proof}

In other words, the basic semilattices of this section are the same as the basic meet semilattices defined at the end of \autoref{SS}.  Part of the above theorem could also be obtained from the Hofmann-Lawson theorem from \cite{HofmannLawson1978}.  Specifically, as $B$ is isomorphic to a $\bigvee$-dense meet subsemilattice of the continuous frame $(\mathcal{P}(B)^\cup,\subset)$, by \autoref{semicor}, $B$ must be isomorphic to a $\cap$-closed basis of a locally compact sober space.

As with basic lattices, the basic semilattice axioms can be much simplified when $\prec$ is reflexive.  Specifically, for a meet semilattice $(B,\preceq)$ to be a basic semilattice, it suffices to omit \eqref{psin} and satisfy $\theta_1$, which becomes
\[\label{Separativity}\tag{Separativity}\nexists v\neq0\ (x\succeq v\perp y)\quad\Rightarrow\quad x\preceq y.\]
While this is still not a finite axiomatization, we show in \autoref{TTS} that more general `pseudobases' of compact clopen sets can be axiomatized by \eqref{Separativity} alone.

\section{Tight Representations}\label{T}

We call a poset $(B,\preceq)$ with minimum $0$ a \emph{p0set} and apply all our previous notation and terminology to p0sets by taking $\prec\ =\ \preceq$.  For $C\subseteq B$, let
\[C_\succeq=\bigcap_{c\in C}\{c\}^\succeq=\{x\in B:\forall y\in C(y\succeq x)\}.\]
We take the empty intersection to be the entire p0set, i.e. $\emptyset_\succeq=B$.

\begin{dfn}\label{covrel}
For $C,D\subseteq B$ we define the \emph{covering relation} $\precapprox$ by
\[C\precapprox D\qquad\Leftrightarrow\qquad C_\succeq\subseteq D^\Cap\cup\{0\}.\]
\end{dfn}

Unlike the other relations we have been considering, $\precapprox$ need not be transitive.  However, $\precapprox$ is at least reflexive on $\mathcal{P}(B)\backslash\{\emptyset\}$.  Also, $\precapprox$ can often be expressed in more familiar order theoretic terms, e.g.
\begin{align*}
&\text{If $B$ is any p0set then}&\{x\}&\precapprox D\quad&&\Leftrightarrow\quad &\{x\}&\precsim D.\\
&\text{If $B$ is separative then}&\{x\}&\precapprox\{y\}\quad&&\Leftrightarrow\quad &x&\preceq y.\\
&\text{If $B$ is a meet semilattice then}&\{x,y\}&\precapprox D\quad&&\Leftrightarrow\quad&\{x\wedge y\}&\precapprox D.\\
&\text{If $B$ is a distributive lattice then}&C&\precapprox\{x,y\}\quad&&\Leftrightarrow\quad&C&\precapprox\{x\vee y\}.
\end{align*}
Also note the following relationships between $\precsim$ and $\precapprox$.
\begin{align*}
B\precsim D\qquad&\Leftrightarrow\qquad\hspace{3pt}\emptyset\precapprox D.\\
\emptyset\neq C\precsim D\qquad&\Rightarrow\qquad C\precapprox D.\\
C\precapprox E\precsim D\qquad&\Rightarrow\qquad C\precapprox D.
\end{align*}

\begin{dfn}\label{tightdef}
If $A$ and $B$ are p0sets and $\beta:B\rightarrow A$ satisfies $\beta(0)=0$ then $\beta$ is
\begin{enumerate}
\item \emph{tight} if $\beta$ preserves $\precapprox$ on $\mathcal{F}(B)$, i.e. if for all $F,G\in\mathcal{F}(B)$,
\begin{equation}\label{FpG}
F\precapprox G\qquad\Rightarrow\qquad\beta[F]\precapprox\beta[G].
\end{equation}
\item \emph{tightish} if \eqref{FpG} holds when $F\neq\emptyset$.
\item \emph{coinitial} if $\beta[B]$ is $\preceq$-coinitial in $A$, i.e. $A\backslash\{0\}=(\beta[B]\backslash\{0\})^\preceq$.
\item a \emph{representation} if $A$ is a generalized Boolean algebra.
\item a \emph{character} if $A=\{0,1\}$.
\end{enumerate}
\end{dfn}

The difference between tight and tightish is illustrated as follows.
\begin{ex}
Let $B=\{0,x,y\}$ be the meet semilattice with $x\wedge y=0$.  So $\emptyset\precapprox\{x,y\}$ is the only non-trivial covering relation.  Thus any $\beta:B\rightarrow A$ with $\beta(0)=0$ is tightish, while $\beta$ is a tight representation iff $A$ is a Boolean algebra with maximum $\beta(x)\vee\beta(y)$.  For example, $\beta$ is not tight when we define $\beta:B\to \mathcal{P}(\{1,2,3\})$ by
\[\beta(0) = \emptyset,\ \beta(x) = \{1\},\ \beta(y) = \{2\}.\]
However, note that if we restrict the codomain to $\mathcal{P}(\{1,2\})$ then $\beta$ is tight.
\end{ex}

In general, we see that
\[
\text{tightish and coinitial}\qquad\Rightarrow\qquad\text{tight}\qquad\Rightarrow\qquad\text{tightish},
\]
and they all coincide if we restrict the codomain to $\beta[B]$.  Also if $B^1$ denotes $B$ with maximum $1$ adjoined and $\beta^1$ denotes the extension of $\beta$ to $B^1$ with $\beta(1)=1(\in A^1)$,
\[\beta\text{ is tight}\qquad\Leftrightarrow\qquad\beta^1\text{ is tightish}.\]
If there is no $G\in\mathcal{F}(B)$ with $\emptyset\precapprox G$ then tight and tightish again coincide.  Even when we do have $G\in\mathcal{F}(B)$ with $\emptyset\precapprox G$, to verify that tightish $\beta:B\rightarrow A$ is tight we only need to check that $\emptyset\precapprox\beta[G]$ for some (rather than all) such $G$.

\begin{prp}\label{betatight}
If $\beta$ is tightish and $\emptyset\precapprox\beta[G]$, for some $G\in\mathcal{F}(B)$, $\beta$ is tight.
\end{prp}

\begin{proof}
For any $C,D\subseteq B$,
\[C\precsim D\qquad\Leftrightarrow\qquad\forall x\in C\ (\{x\}\precapprox D).\]
Thus any tightish $\beta$ also preserves $\precsim$ on $\mathcal{F}(B)$.  If $\emptyset\precapprox F$ then $G\subseteq B\precsim F$ which means $G\precsim F$ so, by $\precsim$-preservation, $\emptyset\precapprox\beta[G]\precsim\beta[F]$ and hence $\emptyset\precapprox\beta[F]$.
\end{proof}

Here, `tight' generalizes \cite[Definition 11.6]{Exel2008} (and \autoref{betatight} generalizes \cite[Lemma 11.7]{Exel2008}) while `tightish' generalizes `cover-to-join' from \cite{DonsigMilan2014}.  The original definitions were restricted to (even Boolean) representations of meet semilattice $B$, in which case we have the following alternative description.

\begin{prp}\label{OriginalTight}
A representation $\beta$ of a meet semilattice $B$ is tightish iff
\begin{equation}\label{RestrictedTight}
\beta(x\wedge y)=\beta(x)\wedge\beta(y)\qquad\text{and}\qquad G\preceq\{x\}\precsim G\ \Rightarrow\ \beta(x)\preceq\bigvee\beta[G],
\end{equation}
for all $x,y\in B$ and $G\in\mathcal{F}(B)$.  Also, $\beta$ is tight iff moreover, for all $G\in\mathcal{F}(B)$,
\begin{equation}\label{1A}
\emptyset\precapprox G\qquad\Rightarrow\qquad1_A=\bigvee\beta[G].
\end{equation}
\end{prp}

\begin{proof}
Assume $\beta$ is tightish.  Then $\beta$ is order preserving because $A$ is separative so
\[x\preceq y\ \Rightarrow\ \{x\}\precapprox\{y\}\ \Rightarrow\ \{\beta(x)\}\precapprox\{\beta(y)\}\ \Rightarrow\ \beta(x)\preceq\beta(y).\]
Thus $\beta(x\wedge y)\preceq\beta(x)\wedge\beta(y)$.  Conversely, by the definition of meets, we have $\{x,y\}_\succeq=\{x\wedge y\}_\succeq\subseteq\{x\wedge y\}^\Cap\cup\{0\}$, i.e. $\{x,y\}\precapprox\{x\wedge y\}$ so $\{\beta(x),\beta(y)\}\precapprox\{\beta(x\wedge y)\}$ and hence $\beta(x)\wedge\beta(y)\preceq\beta(x\wedge y)$.  As $A$ is also distributive, we immediately see that $(G\preceq)\{x\}\precsim G$ implies $\beta(x)\preceq\bigvee\beta[G]$.

On the other hand, if \eqref{RestrictedTight} holds and $F\precapprox G$, for $F,G\in\mathcal{F}(B)$, then we have $G\wedge x\preceq\{x\}\precsim G\wedge x$, for $x=\bigwedge F$, so
 $\beta(x)\preceq\bigvee\beta[G\wedge x]\preceq\bigvee\beta(G)$.

Lastly, for \eqref{1A} note that $A\precapprox\beta[G]$ iff $A\precapprox\{\bigvee\beta[G]\}$, as $A$ is distributive.  As $A$ is separative, this is saying $\bigvee\beta[G]$ is the maximum $1_A$ of $A$.
\end{proof}

In particular, if there is $G\in\mathcal{F}(B)$ with $\emptyset\precapprox G$ then all tight representations of $B$ must be to true Boolean algebras, as in \cite[Definition 11.6]{Exel2008}.  Actually, if we were being faithful to \cite[Definition 11.6]{Exel2008}, we would define
\[B^{C,D}=C_{\succeq} \cap D_{\perp}=\{e\in B: \forall x\in C(x\succeq e) \text{ and }\forall y\in D(y\perp e)\}\]
and call $\beta$ tight if $\beta(0)=0$ and, for all $F,G,H\in\mathcal{F}(B)$,  
\begin{equation}B^{F,G}\precsim H\qquad\Rightarrow\qquad A^{\beta[F],\beta[G]}\precsim\beta[H].
\end{equation}
However, this is equivalent to our definition as
\[B^{F,G}\precsim H\qquad\Leftrightarrow\qquad F\precapprox G\cup H.\]

If we restrict further to generalized Boolean algebra $B$, we see that the tightish representations are precisely the generalized Boolean homomorphisms, i.e. the maps preserving $\wedge$, $\vee$ and $\backslash$.  Indeed, we will soon see how the category of posets with tightish morphisms is in some sense a pullback of the category of generalized Boolean algebras with generalized Boolean morphisms.

\begin{prp}\label{genBoolhomo}
For generalized Boolean algebras $A$ and $B$ and $\beta:B\rightarrow A$, the following are equivalent.
\begin{enumerate}
\item $\beta$ is tightish.
\item $\beta$ is a lattice homomorphism with $\beta(0)=\beta(0)$.
\item $\beta$ is a generalized Boolean homomorphism.
\end{enumerate}
\end{prp}

\begin{proof}
By the observations after \autoref{covrel}, in any generalized Boolean algebra,
\[\{w,x\}\precapprox\{y,z\}\qquad\Leftrightarrow\qquad w\wedge x\ \preceq\ y\vee z.\]
Thus, arguing as in the proof of \autoref{OriginalTight}, we see that the tight maps between generalized Boolean algebras are precisely the lattice homomorphisms taking $0$ to $0$.  As $x\backslash x=0$ and $x\backslash y$ is the unique complement of $x\wedge y$ in $[0,x]$, these are precisely the generalized Boolean homomorphisms.
\end{proof}

This and \eqref{1A} yields the following version of \cite[Proposition 11.9]{Exel2008}.

\begin{prp}
For Boolean algebras $A$ and $B$ and $\beta:B\rightarrow A$, the following are equivalent.
\begin{enumerate}
\item $\beta$ is tight.
\item $\beta$ is a lattice homomorphism with $\beta(0)=\beta(0)$ and $\beta(1)=\beta(1)$.
\item $\beta$ is a Boolean homomorphism.
\end{enumerate}
\end{prp}

\section{The Enveloping Boolean Algebra}

Next we construct a tight map from any given p0set $B$ to what might be called its `enveloping Boolean algebra' $\mathcal{RO}(B')$.  We then examine its universal properties.

First, let $B'=B\backslash\{0\}$ with the Alexandroff topology, where the closed sets are precisely the $\succeq$-closed sets, and consider the map $x\mapsto\{x\}^\succeq\backslash\{0\}$ from $B$ to $\mathcal{O}(B')$.

\begin{prp}\label{OpenTight}
The map $x\mapsto\{x\}^\succeq\backslash\{0\}$ is tight and coinitial.
\end{prp}

\begin{proof}
Take $F,G\in\mathcal{F}(B)$ with $F\precapprox G$.  If $O\in\mathcal{O}(B')$ with $O\subseteq\{x\}^\succeq$, for all $x\in F$, then $O\subseteq F^\succeq$ so $O\subseteq G^\Cap\backslash\{0\}$.  Thus we have $y\in G$ with $\emptyset\neq O\cap\{y\}^\succeq\backslash\{0\}\in\mathcal{O}(B')$ so $O\C\{y\}^\succeq\backslash\{0\}$.  Thus $\{\{x\}^\succeq\backslash\{0\}:x\in F\}\precapprox\{\{y\}^\succeq\backslash\{0\}:y\in G\}$.
\end{proof}

For any topological space $X$, recall that $O\in\mathcal{O}(X)$ is \emph{regular} if $O=\overline{O}^\circ$ or, equivalently, if $O=\overline{N}^\circ$ for any $N\subseteq X$.  The regular open sets $\mathcal{RO}(X)$ form a complete Boolean algebra w.r.t. $\subseteq$ such that, for $O\in\mathcal{RO}(X)$ and $\mathcal{N}\subseteq\mathcal{RO}(X)$,
\[\neg O=(X\setminus O)^\circ\qquad\text{and}\qquad\bigvee\mathcal{N}=\overline{\bigcup\mathcal{N}}^\circ.\]

\begin{prp}\label{OtoRO}
$O\mapsto\overline{O}^\circ$ is a tight coinitial representation of $\mathcal{O}(X)$ in $\mathcal{RO}(X)$.  \end{prp}

\begin{proof}
We first claim that $O\mapsto\overline{O}^\circ$ preserves meets.  For, given any $O,N\in\mathcal{O}(X)$,
\[\overline{O\cap N}^\circ\subseteq(\overline{O}\cap\overline{N})^\circ=\overline{O}^\circ\cap\overline{N}^\circ.\]
For the reverse inclusion, it suffices to show that $\overline{O}^\circ\cap\overline{N}^\circ\subseteq\overline{O\cap N}$ as taking interiors then yields $\overline{O}^\circ\cap\overline{N}^\circ\subseteq\overline{O\cap N}^\circ$.  If this inclusion failed, we would have $\emptyset\neq P=\overline{O}^\circ\cap\overline{N}^\circ\setminus\overline{O\cap N}\in\mathcal{O}(X)$.  As $P\subseteq\overline{O}$, $\emptyset\neq P\cap O\in\mathcal{O}(X)$.  Likewise, $P\cap O\subseteq P\subseteq\overline{N}$ so $\emptyset\neq P\cap O\cap N$, contradicting the definition of $P$.

Also, for $O\in\mathcal{O}(X)$ and $F\in\mathcal{F}(\mathcal{O}(X))$, $\{O\}\precapprox F$ means $O\subseteq\overline{\bigcup F}$ and hence $O\subseteq\overline{\bigcup F}^\circ=\bigvee_{O\in F}O$.  Thus $O\mapsto\overline{O}^\circ$ is tight, by \autoref{OriginalTight}.
\end{proof}

As in \cite[Ch II Lemma 3.3]{Kunen1980}, define $\rho:B\rightarrow\mathcal{RO}(B')$ by
\begin{equation}\label{betadef}
\rho(x)=\overline{\{x\}^\succeq}^\circ.
\end{equation}
By \autoref{OpenTight} and \autoref{OtoRO}, $\rho$ is tight.  In fact, more can be said.

\begin{prp}\label{FGrho}
For all $F,G\in\mathcal{F}(B)$,
\[F\precapprox G\qquad\Leftrightarrow\qquad\bigwedge\rho[F]\subseteq\bigvee\rho[G].\]
\end{prp}

\begin{proof}
For any $Y\subseteq B'$, we see that
\begin{align*}
\overline{Y}&=Y^\preceq.\\
Y^\circ&=\{y\in B':\{y\}^\succeq\backslash\{0\}\subseteq Y\}.\\
Y^\Cap&=(Y^{\succeq}\setminus\{0\})^{\preceq}.\\
\text{Thus}\quad\bigvee\rho[G]&=\overline{\bigcup G^\succeq\backslash\{0\}}^\circ\\
&=\{y\in B':\{y\}^\succeq\subseteq (G^{\succeq}\setminus\{0\})^{\preceq}\}\\
&=\{y\in B':\{y\}^\succeq\subseteq G^\Cap\cup\{0\}\}.
\end{align*}
Also, as $O\mapsto\overline{O}^\circ$ is meet preserving, $\bigwedge\rho[F]=\bigcap_{x\in F}\overline{\{x\}^\succeq\backslash\{0\}}^\circ=\overline{F_\succeq\backslash\{0\}}^\circ$.  If $O\in\mathcal{RO}(B')$ then $\overline{F_\succeq\backslash\{0\}}^\circ\subseteq O$ iff $F_\succeq\backslash\{0\}\subseteq O$.  Thus, as $(F_\succeq)^\succeq=F_\succeq$,
\[\bigwedge\rho[F]\subseteq\bigvee\rho[G]\quad\Leftrightarrow\quad F_\succeq\backslash\{0\}\subseteq G^\Cap\quad\Leftrightarrow\quad F\precapprox G.\qedhere\]
\end{proof}

Thus a representation $\beta$ of a p0set $B$ is tight iff, for $F,G\in\mathcal{F}(B)$,
\[\bigwedge\rho[F]\subseteq\bigvee\rho[G]\qquad\Rightarrow\qquad\bigwedge\beta[F]\preceq\bigvee\beta[G].\]
We now show that $\rho$ restricted to the generalized Boolean subalgebra of $\mathcal{R}\mathcal{O}(B')$ generated by $\rho[B]$ is universal for tight(ish) representations.

\begin{thm}\label{Universality}
Let $\beta:B\to A$ be a representation of a p0set $B$, let $\rho$ be as in \eqref{betadef}, and let $S$ be the generalized Boolean subalgebra of $\mathcal{R}\mathcal{O}(B')$ generated by $\rho[B]$. Then $\beta$ is a tight(ish) representation iff $\beta$ factors through $\rho$, i.e. iff there is tight(ish) $\pi$ from $S$ to $A$ such that $\beta=\pi\circ\rho$.
\end{thm}

\begin{proof}
As $\rho$ is tight, if $\pi$ is tight(ish) then so is $\pi\circ\rho$.

Conversely, assume that $\beta$ is a tightish representation of $B$ in $A$.  In particular, if $\rho(x)=\rho(y)$ then $\beta(x)=\beta(y)$ so we can define $\pi:\rho[B]\rightarrow A$ by
\[\pi(\rho(x))=\beta(x).\]
We can then extend $\pi$ to the meet semilattice $M$ generated by $\rho[B]$ by defining $\pi(\bigwedge\rho[F])=\bigwedge\beta[F]$, for $F\in\mathcal{F}(B)$.  For if $\bigwedge\rho[F]=\bigwedge\rho[G]$ then $\bigwedge\rho[F]\subseteq\rho(y)$, for all $y\in G$, so, by \autoref{FGrho}, $\bigwedge\beta[F]\preceq\beta(y)$, as $\beta$ is tightish.  This means $\bigwedge\beta[F]\preceq\bigwedge\beta[G]$ and, by a dual argument, $\bigwedge\beta[F]\preceq\bigwedge\beta[G]$.  It the follows from the defintion that this extension to $M$ is meet preserving.

As $\mathcal{RO}(B')$ is distributive, the lattice $L$ generated by $\rho[B]$ is generated by joins of elements of $M$.  We claim we can extend $\pi$ to $L$ by defining, for $F\in\mathcal{F}(M)$,
\[\pi(\bigvee F)=\bigvee\pi[F].\]
For if $\bigvee F=\bigvee G$ then, for all $O\in F$, we have $H\in\mathcal{F}(B)$ with $\bigwedge\rho[H]=O$, and, for all $N\in G$, we have $H_N\in\mathcal{F}(B)$ with $\bigwedge\rho[H_N]=N$.  So if $x_N\in H_N$, for all $N\in G$, then $\bigwedge\rho[H]=O\subseteq\bigvee G\subseteq\bigvee_{N\in G}\rho(x_N)$ and hence \autoref{FGrho} and tightishness yields $\bigwedge\beta[H]\preceq\bigvee_{N\in G}\beta(x_N)$.  Thus distributivity yields
\begin{align*}
\pi(O)&=\pi(\bigwedge\rho[H])\\
&=\bigwedge\beta[H]\\
&\preceq\bigwedge_{\{(x_N):\forall N\in G(x_N\in H_N)\}}\bigvee_{N\in G}\beta(x_N)\\
&=\bigvee_{N\in G}\bigwedge\beta[H_N]\\
&=\bigvee_{N\in G}\pi[\bigwedge\rho[H_N]]\\
&=\bigvee\pi[G].
\end{align*}
Therefore $\bigvee\pi[F]\preceq\bigvee\pi[G]$ and, again by a dual argument, $\bigvee\pi[G]\preceq\bigvee\pi[F]$.

For any sublattice $L$ of $S$ and any $x\in L$, let $L_x$ be the sublattice
\[L_x=\{y\vee(z\backslash x):y,z\in L\}.\]
As $w=(w\wedge x)\vee(w\backslash x)$, for all $w\in S$, $L\subseteq L_x$ and it suffices to take $y\subseteq x$ above.  Then we claim that any lattice homomorphism $\pi$ from $L$ to $A$ can be extended to a lattice homomorphism $\pi'$ of $L_x$ given by
\[\pi'(y\vee(z\backslash x))=\pi(y)\vee(\pi(z)\backslash\pi(x)).\]
To see that this is well-defined, say $y\vee(z\backslash x)=y'\vee(z'\backslash x)$ with $y,y'\subseteq x$.  Then $y=x\wedge y=x\wedge(y\vee(z\backslash x))=x\wedge(y'\vee(z'\backslash x))=x\wedge y'=y'$.  Likewise $z\backslash x=z'\backslash x$, which is equivalent to $x\vee z=x\vee z'$.  As $\pi$ is a lattice homomorphism, $\pi(x)\vee\pi(z)=\pi(x)\vee\pi(z')$ and hence $\pi(z)\backslash\pi(x)=\pi(z')\backslash\pi(x)$.  
Also, for any $w\in L_x$, $w=(w\wedge x)\vee(w\backslash x)$ yields
\[\pi'(w)=\pi(w\wedge x)\vee(\pi(w)\backslash\pi(x))=(\pi(w)\wedge\pi(x))\vee(\pi(w)\backslash\pi(x))=\pi(w).\]
So $\pi'$ extends $\pi$ and likewise $\pi'$ is verified to be a lattice homomorphism.

Thus any maximal lattice homomorphism extension of $\pi$ defined on the sublattice generated by $\rho[B]$ as above must in fact be defined on the entirety of $S$.  Thus $\pi$ is tightish, by \autoref{genBoolhomo}.  If there is no $G\in\mathcal{F}(B)$ with $\emptyset\precapprox G$ then $\pi$ is even (vacuously) tight.  While if $G\in\mathcal{F}(B)$, $\emptyset\precapprox G$ and $\beta$ is tight then $\emptyset\precapprox\beta[G]=\pi\circ\rho[G]$ so $\pi$ is also tight, by \autoref{betatight}.
\end{proof}

It follows that tight(ish) maps between general p0sets are precisely those coming from tight(ish) maps of generalized Boolean algebras.

\begin{cor}
For p0sets $A$ and $B$ with $\rho_A:A\rightarrow S_A$ and $\rho_B:B\rightarrow S_B$ as above, $\beta:B\rightarrow A$ is tight(ish) iff $\rho_A\circ\beta=\pi\circ\rho_B$ for some tight(ish) $\pi:S_B\rightarrow S_A$.
\end{cor}

\begin{proof}
If $\beta$ is tight(ish) then so is $\rho_A\circ\beta$ and the required $\pi$ comes from \autoref{Universality}.  On the other hand, if $\rho_A\circ\beta=\pi\circ\rho_B$ and $\pi$ is tight(ish) then so is $\pi\circ\rho_B$ and hence $\rho_A\circ\beta$.  This means, for all $F,G\in\mathcal{F}(B)$ (with $F\neq\emptyset$), 
\[F\precapprox G\quad\Rightarrow\quad\rho_A\circ\beta[F]\precapprox\rho_A\circ\beta[G]\quad\Leftrightarrow\quad\beta[F]\precapprox\beta[G],\]
by \autoref{FGrho}, so $\beta$ is tight(ish) too.
\end{proof}

Thus we have a map $\beta\mapsto\pi_\beta$ taking any tight(ish) $\beta:B\rightarrow A$ to the unique tight(ish) $\pi_\beta:S_B\rightarrow S_A$ satisfying $\rho_A\circ\beta=\pi_\beta\circ\rho_B$.  To put this in category theory terms, let $\mathbf{P}$ denote the category of p0sets with tight(ish) morhpisms and let $\mathbf{G}$ denote its full subcategory of generalized Boolean algebras.  The above results are saying that we have a full functor $F$ from $\mathbf{P}$ onto $\mathbf{G}$ with $F(B)=S_B$ and $F(\beta)=\pi_\beta$ together with a natural transformation $\rho$ from the identity functor $I$ to $F$:
\[\begin{tikzcd}
B \arrow{r}{\beta} \arrow[swap]{d}{\rho_B} & A \arrow{d}{\rho_A} \\
S_B \arrow{r}{\pi_\beta} & S_A
\end{tikzcd}\]

\section{The Tight Spectrum}\label{TTS}

\begin{dfn}
Let $B$ be any p0set. The \emph{tight spectrum} $\check{B}$ is the space of non-zero tight characters on $B$ taken as a subspace of $\{0,1\}^B$ with the product topology.
\end{dfn}

We could equivalently call this the tightish spectrum, as any non-zero tightish character is coinitial and hence tight.  And if $\emptyset\precapprox G$, for some $G\in\mathcal{F}(B)$, then every tight character is automatically non-zero and so our definition of $\check{B}$ agrees with the definition of $\hat{B}_\mathrm{tight}$ from \cite[Definition 12.8]{Exel2008}.  When there is no $G\in\mathcal{F}(B)$ with $\emptyset\precapprox G$, we instead have $\hat{B}_\mathrm{tight}=\check{B^1}=$ the one-point compactification of $\check{B}$, where $B^1$ here denotes $B$ with a top element $1$ adjoined.

We can also view the tight spectrum as a certain Stone space, for by \autoref{Universality}, we can identify $\check{S}$ and $\check{B}$ via the map $\phi\mapsto\phi\circ\rho$.  We can then identify $\check{S}$ with $\hat{S}$ via the map $\phi\mapsto\phi^{-1}\{1\}$, as non-zero tight characters on generalized Boolean algebras are precisely the characteristic functions of ultrafilters (as lattice homomorphisms from generalized Boolean algebras to $\{0,1\}$ are precisely the characteristic functions of prime filters).

\begin{dfn}\label{pseudobasisdef}
For any topological space $X$, we call $B\subseteq\mathcal{O}(X)$ a \emph{pseudobasis} if
\begin{align}
\label{Minimum2}\tag{Minimum}&\emptyset\in B.\\
\label{Cover}\tag{Cover}&X=\bigcup B.\\
\label{Coinitiality2}\tag{Coinitiality}&\emptyset\neq O\in\mathcal{O}(X)\ \Rightarrow\ \exists N\in B\ (\emptyset\neq N\subseteq O).\\
\label{T0}\tag{$T_0$}&\forall x,y\in X\ \exists O\in B\ (x\notin O\ni y\text{ or }y\notin O\ni x).
\end{align}
\end{dfn}

For $x\in B$, let $O_x=\{\phi\in\check{B}:\phi(x)=1\}$.  Also let $\dot{B}$ denote the characteristic functions of maximal centred $C\subseteq B$, i.e. satisfying $F_\succeq\neq\{0\}$, for all $F\in\mathcal{F}(C)$.

\begin{prp}\label{checkB}
$\check{B}$ is $0$-dimensional locally compact Hausdorff with pseudobasis $(O_x)_{x\in B}$ and dense subset $\dot{B}$.
\end{prp}

\begin{proof}
As $\{0,1\}^B$ is $0$-dimensional Hausdorff, so is $\check{B}$.  The tight characters are immediately seen to form a closed subset of $\{0,1\}^B$ so taking away the zero character still yields a locally compact space $\check{B}$.

As $O_0=\emptyset$, $(O_x)_{x\in B}$ satisfies \eqref{Minimum2}.  As every $\phi\in\check{B}$ has value $1$ for some $x\in X$, $(O_x)_{x\in B}$ satisfies \eqref{Cover}.  If $\phi,\psi\in\check{B}$ are distinct then $\phi(x)\neq\psi(x)$, for some $x\in B$, so $(O_x)_{x\in B}$ satisfies \eqref{T0}.  To see that $(O_x)$ satisfies \eqref{Coinitiality2}, take some non-empty basic open $O\subseteq\check{B}$, so we have $F,G\in\mathcal{F}(B)$ with
\[O=\{\phi\in\hat{B}:\phi[F]=\{1\}\text{ and }\phi[G]=\{0\}\}.\]  For any $\phi\in O$, $\phi[F]\not\precapprox\phi[G]$ so $F\not\precapprox G$, by tightness.  But this means we have non-zero $x\in F_\succeq\backslash G^\Cap$.  For any $\psi\in O_x$, $\psi[F]=\{1\}$, as $\psi$ is order preserving.  If $x\perp y$ then $\{x,y\}\precapprox0$ so $\{\psi(x),\psi(y)\}\precapprox\psi(0)=0$ and hence $\psi(x)\perp\psi(y)$, i.e. $\psi$ is also orthogonality preserving and hence $\psi[G]=\{0\}$ so $O_x\subseteq O$.  It only remains to show that $O_x$ is non-empty.  So let $\phi$ be the characteristic function of some maximal centred $C\subseteq B$ containing $x$.  If $\phi$ were not tight, then we would have $F,G\in\mathcal{F}(B)$ with $\phi[F]=\{1\}$, $\phi[G]=\{0\}$ and $F\precapprox G$.  As $C$ is maximal, for each $x\in G$ we have $H_x\in\mathcal{F}(C)$ with $(\{x\}\cup H_x)_\succeq=\{0\}$.  But then $(F\cup\bigcup H_x)_\succeq=\{0\}$, contradicting the fact $C$ is centred.  Thus $\phi\in O_x$ and this also shows that $\dot{B}$ is dense in $\check{B}$.
\end{proof}

\begin{thm}
If $B$ is a pseudobasis of compact clopen subsets of a topological space $X$ then we have a homeomorphism from $X$ onto $\check{B}$ given by $x\mapsto\phi_x$ where
\[\phi_x(O)=\begin{cases}1 &\text{if }x\in O\\ 0&\text{if }x\notin O.\end{cases}\]
\end{thm}

\begin{proof}
For any $F,G\in\mathcal{F}(B)$, we claim that
\[F\precapprox G\qquad\Leftrightarrow\qquad\bigcap F\subseteq\bigcup G.\]
If $\bigcap F\subseteq\bigcup G$ then, for any non-empty $O\in F_\subseteq$, we have $O\subseteq\bigcap F\subseteq\bigcup G$ so $\emptyset\neq O\cap\bigcup G=\bigcup_{N\in G}O\cap N$.  Thus $\emptyset\neq O\cap N$, for some $N\in G$, and hence $O\C N$ in $B$, by \eqref{Coinitiality2}, i.e. $O\in G^\Cap$ so $F\precapprox G$.  Conversely, if $\bigcap F\nsubseteq\bigcup G$ then $\emptyset\neq\bigcap F\setminus\bigcup G$ so \eqref{Coinitiality2} yields non-empty $O\in B$ with $O\subseteq\bigcap F\setminus\bigcup G$, i.e. $O\in F_\subseteq\setminus(G^\Cap\cup\{\emptyset\})$ so $F\not\precapprox G$.  Thus, for all $x\in X$, the definition of $\phi_x$ yields
\begin{align*}
F\precapprox G\qquad\qquad&\Leftrightarrow&\bigcap F\ &\subseteq\ \bigcup G.\\
&\Rightarrow&x\in\bigcap F\ &\Rightarrow\ x\in \bigcup G.\\
&\Rightarrow&\forall O\in F\ \phi_x(O)=1\ &\Rightarrow\ \exists N\in G\ \phi_x(N)=1\\
&\Rightarrow&\bigwedge\phi_x[F]\ &\preceq\ \bigvee\phi_x[G]\\
&\Rightarrow&\phi_x[F]\ &\precapprox\ \phi_x[G].
\end{align*}
Thus $\phi_x$ is tight and also non-zero, by \eqref{Cover}, so $\phi_x\in\check{B}$.

We next claim that, for any $\phi\in\check{B}$, there is a unique $\{x\}$ such that
\[\{x\}=\bigcap_{\substack{\phi(O)=1\\ \phi(N)=0}}O\backslash N.\]
As $\phi\neq0$ and the elements of $B$ are compact clopen, if the intersection were empty then it would be empty for some finite subset, i.e. we would have $F,G\in\mathcal{F}(B)$ with $\phi[F]=\{1\}$, $\phi[G]=\{0\}$ and $\emptyset=\bigcap_{O\in F,N\in G}O\backslash N$.  But this means $\bigwedge\phi[F]=1\npreceq0=\bigvee\phi[G]$ and $\bigcap F\subseteq\bigcup G$, contradicting the tightness of $\phi$.  On other other hand, the intersection can not contain more than one point, by \eqref{T0}.  This proves the claim, which means $\phi=\phi_x$.  Thus $x\mapsto\phi_x$ is a bijection from $X$ to $\check{B}$.

Now say we have $x\in M\in\mathcal{O}(X)$.  By \eqref{T0},
\[\emptyset=\bigcap_{\substack{x\in O\in B\\ x\notin N\in B\\\text{or }N=M}}O\backslash N.\]
As each $O\in B$ is compact clopen and $B$ satisfies \eqref{Cover}, some finite subset has empty intersection, i.e. we have $F,G\in\mathcal{F}(B)$ with $x\in\bigcap F$, $x\notin\bigcup G$ and $\bigcap_{O\in F,N\in G}O\backslash N\subseteq M$.  As $x$ and $M$ were arbitrary, this is saying $x\mapsto\phi_x$ is an open mapping.  As each $O\in B$ is clopen, $x\mapsto\phi_x$ is also continuous and hence a homeomorphism.
\end{proof}

\begin{cor}\label{pseudochar}
Separative p0sets characterize compact clopen pseudobases of necessarily $0$-dimensional locally compact Hausdorff topological spaces.  More precisely: If $B$ is a compact clopen pseudobasis of $X$ then $(B,\subseteq)$ is separative and $\check{B}$ is homeomorphic to $X$, while if $(B,\preceq)$ is separative then $\check{B}$ has compact clopen pseudobasis $(O_x)_{x\in B}$ order isomorphic to $B$.
\end{cor}

\begin{proof}
If $B$ is a compact clopen pseudobasis of $X$ then, for any $O,N\in B$ with $O\nsubseteq N$, we see that $\emptyset\neq O\backslash N\in\mathcal{O}(X)$.  By \eqref{Coinitiality2}, we then have non-empty $M\in B$ with $M\subseteq O\backslash N$, so $B$ is separative.  By \autoref{checkB}, $X$ is homeomorphic to $\check{B}$, which is $0$-dimensional locally compact Hausdorff, by \autoref{checkB}.

If $(B,\preceq)$ is separative then, whenever $x\not\preceq y$, we have non-zero $z\preceq x$ with $y\perp z$.  We can then take $\phi\in\dot{B}$ with $\phi(z)=1$ so $\phi\in O_x\backslash O_y$ and hence $O_x\nsubseteq O_y$.  Conversely, if $x\preceq y$ then $O_x\subseteq O_y$ so
\begin{equation}\label{xOxIso}
x\preceq y\qquad\Leftrightarrow\qquad O_x\subseteq O_y.\qedhere
\end{equation}
\end{proof}

We finish with a note on \eqref{Separativity}, which is the standard term in set theory (see \cite[Ch II Exercise (15)]{Kunen1980}), and some other closely related conditions.  First note that \eqref{xOxIso} implies that $x\mapsto O_x$ is injective, as $O_x=O_y$ then implies $x\preceq y\preceq x$ so $x=y$.  This is equivalent to saying that $\rho$ from \eqref{betadef} is injective.  This, in turn, implies that $B$ is `section semicomplemented' in the sense of \cite[Definition 4.17]{MaedaMaeda1970}, specifically
\[\label{SSC}\tag{SSC}x\neq y\preceq x\quad\Rightarrow\quad\exists z\neq0\ (y\perp z\preceq x).\]
Equivalently, this means equality=density from \cite[Definition 11.10]{Exel2008}.  So
\[\eqref{Separativity}\quad\Rightarrow\quad\rho\text{ is injective}\quad\Rightarrow\quad\eqref{SSC},\]
for general p0set $B$, and they are all equivalent when $B$ is a meet semilattice, by \cite[Proposition 9]{AkemannBice2014}.  This resolves the questions in section 7 of \cite{Exel2007} by providing the converse to \cite[Proposition 11.11]{Exel2008}.  When restricted to lattices, \eqref{Separativity} is also sometimes called `Wallman's Disjunction Property', having first appeared in \cite{Wallman1938}, which is also the dual to `subfit' as defined in \cite[Ch V \S1]{PicadoPultr2012}.

\bibliography{maths}{}
\bibliographystyle{alphaurl}

\end{document}